\documentclass[a4paper,10pt]{article}
\usepackage[top=2.54cm,bottom=2.0cm,left=2.0cm,right=2.54cm, includeheadfoot]{geometry}

\usepackage[T1]{fontenc}
\usepackage[utf8]{inputenc}
\usepackage[english]{babel}
\usepackage{enumerate}
\usepackage{multirow,booktabs}
\usepackage[table]{xcolor}
\usepackage{fullpage}
\usepackage{lastpage}
\usepackage{indentfirst}

\usepackage{footnote}

\usepackage{amsmath,amsfonts,amssymb,amscd,amsthm}

\usepackage{bm}

\usepackage[all,2cell]{xy} \UseAllTwocells \SilentMatrices

\usepackage{mathpartir}

\usepackage{hyperref}

\usepackage{subfiles}

\usepackage{multicol}

\newtheorem{theorem}{Theorem}[section]

\newtheorem{proposition}[theorem]{Proposition} 

\newtheorem{definition}[theorem]{Definition} 
\newtheorem{example}[theorem]{Example}

\newtheorem*{claim*}{Claim}

\newtheorem{remark}[theorem]{Remark}

\begin{document}

\title{Hyper Swap Structures: The Case Study of LFIs and Hyper Boolean Algebras}

\author{
 Marcelo E. Coniglio $^\textup{\scriptsize a,b}$
   \and
   Kaique Matias de Andrade Roberto $^\textup{\scriptsize a,b}$ \and Ana Claudia Golzio
 $^\textup{\scriptsize d}$
}

\date{
   $^\textup{\scriptsize a}$\textit{\small Centre for Logic, Epistemology and the History of Science (CLE), State University of Campinas (UNICAMP), Campinas, Brazil}
   \\
   $^\textup{\scriptsize b}$\textit{\small Institute of Philosophy and the Humanities (IFCH),  State University of Campinas (UNICAMP), Campinas, Brazil}
    \\
    $^\textup{\scriptsize c}$\textit{\small School of Sciences and Engineering (FCE),  São Paulo State University (UNESP), Tupã, Brazil}
}

\maketitle

\begin{abstract} 

In a previous paper, we introduced the notion of hyper swap structures, a novel class of hyperalgebras that naturally generalizes swap structures semantics. In this paper we introduce the concept of hyper Boolean algebras based on Morgado hyperlattices, proving some basic properties. From this, we show that several paraconsistent logics in the hierarchy of  Logics of Formal Inconsistency (LFIs) can be naturally characterized in terms of hyper swap structures semantics generated by hyper Boolean algebras. Finally,  for each of these LFIs we obtain a Kalman-style functor which establishes an equivalence between the category of hyper Boolean algebras and a category of hyper algebras for the corresponding LFI having the hyper swap structures as representative objects.\\

\textbf{Keywords:} Swap structures, Logics of Formal Inconsistency, Hyperalgebras, Kalman functors.

\textbf{AMS Classification:} 03B53, 03G10, 03G25, 06E75
\end{abstract}

\section{Introduction}

The genesis of the hyper algebraic structures analyzed in this paper lies in the seminal work of J.~A. Kalman~\cite{kalman:1958}. By considering a bounded distributive lattice $\mathcal{L}=\langle L,\land,\lor,0,1\rangle$, Kalman constructed a centered Kleene algebra $K(\mathcal{L})$ defined over the subset of disjoint pairs $\{(a,b) \in L^2 \ : \ a \land b =0\}$. The algebraic operations are defined as follows:
\[
(a,b) \,\bar{\land}\, (c,d) = (a \land c, b \vee d), \quad
(a,b) \,\bar{\lor}\, (c,d) = (a \lor c, b \land d), \quad
\bar{\neg} (a,b) = (b, a),
\]
and the unique element in the center of $K(\mathcal L)$ is $(0,0)$.

The categorical analysis of this construction was later established by R. Cignoli~\cite{cign:86}. He introduced the \emph{Kalman functor} $K$, between the category of bounded lattices and the category of centered Kleene algebras. For any lattice homomorphism $f:\mathcal{L} \to \mathcal{L}'$, the morphism is extended as $K(f)(a,b)=(f(a),f(b))$. Cignoli proved that $K$ has a left adjoint and, when restricted to algebras satisfying specific interpolation properties (see~\cite{jansana:sanmart:2018}), induces an equivalence of categories. This framework unifies similar constructions for Nelson and Heyting algebras independently proposed by M. Fidel~\cite{fidel:1978} and D. Vakarelov~\cite{vakarelov:1977}. Following M. Kracht~\cite{kracht:1998}, these constructions are collectively referred to as \emph{twist structures}. The associated \emph{twist structure semantics} typically designates the set of truth values as $D = \{(a,b) \ : \ a=1\}$. This methodology has since been adapted to a wide class of logical systems and algebraic contexts (e.g., \cite{odin:08, Rivieccio2014, casti:cel:sanmart:2017, busa:gala:mar:2022}).

While Abstract Algebraic Logic (AAL) provides a powerful toolkit for characterizing logical systems (see~\cite{font:16}), its traditional methods face limitations with systems lacking standard algebraic counterparts, such as the Logics of Formal Inconsistency (LFIs)~\cite{carnielli2016paraconsistent}. Many such systems cannot be characterized by a single finite logical matrix. This limitation necessitates the use of non-deterministic semantics, such as non-truth-functional bivaluations and, crucially, non-deterministic matrices (Nmatrices). Formalized by A. Avron and I. Lev~\cite{Av:01}, Nmatrices generalize standard matrices by employing \emph{hyperalgebras}, where connectives are interpreted as hyperoperators mapping inputs to non-empty sets of values. This approach was anticipated by Yu. Ivlev in the context of non-normal modal logics (see the survey in~\cite{Ivlev:2024}).

A systematic analytical method for defining Nmatrices is provided by \emph{swap structures}, introduced in~\cite[Chapter~6]{carnielli2016paraconsistent}. These can be viewed as non-deterministic variants of twist structures. Their elements are $(n+1)$-tuples over an ordered algebra $\mathcal{A}$ (typically Boolean), where the coordinates encode the truth values of formulas $\alpha_i(\varphi)$ involving non-truth-functional connectives. Swap structures have been successfully applied to Ivlev-like modal logics~\cite{coniglio2019swap} and various LFIs~\cite{coniglio2020non}. In these contexts, a generalized Kalman functor $K$ can be defined from Boolean algebras to swap structures. While this functor preserves products and monomorphisms, the existence of a left adjoint appears structurally unattainable in the non-deterministic setting.

In a recent study~\cite{con:gol:rob:2025}, we extended this approach to characterize da Costa's paraconsistent logic $C_\omega$. Unlike LFIs, $C_\omega$ is not finitely trivializable and extends positive intuitionistic logic ($\mathsf{IPL^+}$). Consequently, our characterization utilized swap structures defined over implicative lattices—the algebraic counterpart of $\mathsf{IPL^+}$. These structures expand upon the hyperlattices of Morgado~\cite{morgado1962introduccao} and Sette~\cite{sette1971algebras}, forming a class of hyperalgebras HC$_\omega$A. We established a Kalman-style functor from the category of implicative lattices to HC$_\omega$A. The logic $C_\omega$ is sound and complete with respect to both the swap structures semantics and the Nmatrix semantics over HC$_\omega$A. Consistent with previous results for LFIs, constructing a left adjoint for this functor remains an elusive goal.

The absence of an adjoint for the Kalman functor in the hyperalgebraic setting (via swap structures), in contrast to the classical algebraic case (via twist structures), can be attributed to a categorical mismatch. In the classical construction, the Kalman functor operates between categories of \emph{algebras}. Consequently, the adjoint functor merely ``forgets'' the additional operators (typically negation) while preserving the underlying algebraic structure. In turn, within the swap structure framework, the functor maps an \emph{algebra} (e.g., an implicative lattice for $C_\omega$, or a Boolean algebra for LFIs) to a \emph{hyperalgebra}. Any hypothetical adjoint attempting to reverse this process would need to map a hyperalgebra back to an algebra. This reduction is problematic because the reduct of a swap structure—stripping away the new operators—remains a hyperalgebra, not an algebra. Forcing this transition inevitably results in a loss of structural information. Therefore, a more natural approach is to define swap structures starting from hyperalgebras that share the same logical behavior as the original algebras. By doing so, the Kalman-style functor connects two categories of hyperalgebras, mirroring the symmetry found in standard twist structure constructions.

Consequently, the primary objective of this paper is to generalize the methodology of \emph{hyper swap structures and enriched hyperalgebras}, originally proposed in~\cite{con:gol:rob:2026}, to the logic mbC and other LFIs that extend it. We aim to establish categorical equivalences between a ``base'' category of hyperalgebras (interpreting standard operators like conjunction, disjunction, and implication) and an ``induced'' category of hyperalgebras (accommodating non-truth-functional operators, such as paraconsistent negation). This mirrors the success of twist constructions in the algebraic realm. As a paradigmatic case, we replace the swap structures for mbC defined over standard Boolean algebras (as seen in~\cite[Chapter~6]{car:con:16}) with hyper swap structures defined over \emph{Hyper Boolean Algebras} (HBAs). By focusing on the subcategory of \emph{enriched hyper mbC algebras} (Definition~\ref{enriched-mbc}), we demonstrate in Section~\ref{sec6} that the category of HBAs is equivalent to that of enriched hyper mbC algebras, with hyper swap structures acting as the representative objects. This generalization—operating entirely within the domain of hyperalgebras—is essential for lifting twist-based equivalences to the non-deterministic context.

The paper is organized as follows: Section~\ref{sec2} briefly reviews the fundamental concepts and results concerning hyperlattices (referencing~\cite{con:gol:rob:2025} for further details). In Section~\ref{sec3}, we introduce the concept of Hyper Boolean Algebras (HBAs) and derive their basic properties. Section~\ref{sec4} provides an overview of the paraconsistent logic mbC, adhering to the presentation in~\cite{car:con:16}. In Section~\ref{sec5}, we formally define hyper mbC algebras and hyper swap structures for mbC, developing the corresponding semantics and proving the soundness and completeness of mbC with respect to this framework. Section~\ref{sec6} defines the class of Enriched Hyper mbC Algebras (EHmbCA) and establishes the core result: the equivalence of categories between HBAs and EHmbCAs. Subsequently, Section~\ref{sec7} adapts these results to other LFIs that are axiomatic extensions of mbC. Finally, Section~\ref{sec8} summarizes our findings and outlines potential avenues for future research.


\section{Preliminaries}\label{sec2}


In this section we recall the basic notions and results on hyperlattices used throughout this paper, which were taken from our previous papers~\cite{coniglio2020non} and~\cite{con:gol:rob:2025}. 

A {\em signature} is a family $\Theta=(\Theta_n)_{n \in \mathbb N}$ of pairwise disjoint sets. Elements of $\Theta_n$ are called {\em $n$-ary connectives}. 

A {\em hyperalgebra} over a signature $\Theta$ is a set $A$ endowed with a family of functions (called {\em $n$-ary hyperoperations}) $\mathcal{O}(\#) : A^n \to \mathcal \wp(A)\setminus\{\emptyset\}$, for every $\# \in \Theta_n$ and $n \in \mathbb N$.\footnote{For simplicity, we will frequently write $\#$ instead of $\mathcal{O}(\#)$ to denote the hyperoperation interpreting the connective $\#$.}

A {\em pre-ordered set} (proset) is a pair $\mathsf P=\langle P, \preceq\rangle$ such that $P$ is a non-empty set and $\preceq$ is a reflexive and transitive relation on $P$. We say that $x$ and $y$ are {\em similar},  denoted by $x \equiv y$, if $x \preceq y$ and $y \preceq x$. For $B,C \subseteq P$, the expression $B \preceq C$ means that $x \preceq y$ for every $x \in B$ and every $y \in C$. 

\begin{definition}
Let $\mathsf P$ be a pre-ordered set, and let $B \subseteq P$.
\begin{enumerate}
    \item The set of {\em minima} of $B$ is $\mathsf{Min}(B)=\{x \in B \ : \ x \preceq B\}$, and the  set of {\em maxima} of $B$ is $\mathsf{Max}(B)=\{x \in B \ : \ B \preceq x\}$.
    \item The set of {\em upper bounds} of $B$ is $\mathsf{Ub}(B)=\{z \in P \ : B \preceq z\}$. The set of lower bounds of $B$ is $\mathsf{Lb}(B)=\{z \in P \ : z \preceq B\}$.
\end{enumerate}
\end{definition}

\begin{definition} [Morgado hyperlattices, {\cite[Ch.~II, \S 2, p.~122]{morgado1962introduccao}}] \label{def:m-hlattices}  Let $\mathsf P$ be a proset, and let $x,y \in P$.
\begin{enumerate}
    \item The {\em Morgado hypersupremum} (or {\em supremoid}) of $x$ and $y$ is the set $x \curlyvee y = \mathsf{Min}(\mathsf{Ub}(\{x,y\}))$.
    \item The {\em Morgado hyperinfimum} (or {\em infimoid}) of $x$ and $y$ is the set $x \curlywedge y = \mathsf{Max}(\mathsf{Lb}(\{x,y\}))$.
    
    \item $\mathsf P$ is said to be a {\em Morgado hyperlattice} (or an {\em m-hyperlattice}, or simply an {\em hyperlattice}) if $x \curlyvee y$ and $x \curlywedge y$ are nonempty sets for every $x,y \in P$.
\end{enumerate}
\end{definition}

\begin{definition} [Stable sets, {\cite[Definition~8]{coniglio2020non}}] \label{stable} Let  $\emptyset\neq A, B \subseteq L$. We say that $A$ and $B$ are {\em similar}, and write $A\equiv B$, if $a \equiv b$ for every $a \in A$ and $b \in B$. That is: $a \preceq b$ and $b \preceq a$  for every $a \in A$ and $b \in B$. A non-empty subset $A\subseteq L$ is {\em stable} if $A\equiv A$, i.e., $x \equiv y$ for every $x,y \in A$.
\end{definition}

\begin{proposition}[Section  2 in \cite{con:gol:rob:2025}]\label{stable-wedge-02}
    Let $A,B,C\subseteq L$ be stable sets, and let $\#,\#' \in\{\curlywedge, \curlyvee\}$. Then, $A \# B$ is stable and, for all $a\in A$, $b\in B$ and $c\in C$:
    \begin{enumerate}
        \item $A \# B= a \# b$.

        \item $A\#(B\#' C)=a\#(b\#' c)$ \ and \ $(A\# B)\#' C=(a\#b)\#' c$.

        \item $A\#(B\# C)=a\#(b\# c)=(a\#b)\# c=(A\# B)\# C$.

        \item $A \# B \preceq C$ iff $a \# b \preceq c$ for some $a \in A$, $b \in B$ and $c \in C$.

        \item $C \preceq A \# B$ iff $c \preceq a \# b$ for some $a \in A$, $b \in B$ and $c \in C$.
    \end{enumerate}
\end{proposition}

\begin{definition}  Let $\mathsf P=\langle P, \preceq, \curlywedge, \curlyvee\rangle$ be a hyperlattice. The sets  $\mathsf{Min}(P)$ and $\mathsf{Max}(P)$ of minima and maxima elements of $P$ will be denoted by $\bot$ and $\top$, respectively.
\end{definition}

\begin{definition} Let $\mathsf P$ be a hyperlattice.\\
(1) We say that $\mathsf P$ is {\em bounded} if $\top\neq\emptyset\neq \bot$.\\
(2) We say that $\mathsf P$ is  {\em distributive} if, for all $x,y,z\in P$,
    $$x\curlywedge(y\curlyvee z)=(x\curlywedge y)\curlyvee(x\curlywedge z) \ \mbox{ and } \ x\curlyvee(y\curlywedge z)=(x\curlyvee y)\curlywedge(x\curlyvee z).$$
\end{definition}

\begin{definition} [Sette implicative hyperlattices, {\cite[Definition~2.3]{sette1971algebras}}] \label{def:s-ilattices}  A {\em Sette implicative hyperlattice} (or an {\em IHL}) is a hyperalgebra \mbox{$\mathsf L=\langle L,\curlywedge,\curlyvee,\multimap  \rangle$} such that the reduct $\langle L,\curlywedge,\curlyvee\rangle$  is a hyperlattice and the hyperoperator $\multimap$  satisfies the following properties, for every $x,y,z, z' \in L$:
\begin{description}
    \item[(I1)] $z \in x \multimap y$ implies that $x \curlywedge z \preceq y$;

    \item[(I2)] $x \curlywedge z \preceq y$ implies that   $z \preceq x \multimap y$;

    \item[(I3)] $z \equiv z'$ and $z \in x \multimap y$ implies that $z' \in x \multimap y$.
\end{description}
\end{definition}

\begin{definition}  [{\cite[Definition~12]{con:gol:rob:2025}}]  Let $\mathsf P$ be a hyperlattice, and let $x,y \in P$. The set $\mathsf R(x,y)$ is given by $\mathsf R(x,y)=\{z \in P \ : \ x \curlywedge z \preceq y\}$.
\end{definition}

\begin{proposition} [{\cite[Proposition~10]{con:gol:rob:2025}}] \label{prop:char:IHL} Let $\mathsf L=\langle L,\curlywedge,\curlyvee,\multimap  \rangle$ be a  hyperalgebra such that $\langle L,\curlywedge,\curlyvee\rangle$  is a hyperlattice. Then, $\mathsf L$ is an  IHL iff $x \multimap y = \mathsf{Max}(\mathsf R(x,y))$, for every $x,y \in L$.
\end{proposition}

\begin{proposition}  [{\cite[Corollary~2]{con:gol:rob:2025}}] 
    Every IHL is a distributive hyperlattice.
\end{proposition}

The following properties will be useful in the proof of Proposition~\ref{prop:HBA-IHL}:

\begin{proposition}   [{\cite[Proposition~13]{con:gol:rob:2025}}] \label{prop:IHL1}
   Let $\mathsf L$ be an IHL. Then, for every $\emptyset \neq A,B,C \subseteq L$:\\[1mm]
   (1) $A  \preceq B$ iff  $A \multimap B = \top$.\\[1mm]
   (2) $A \curlywedge B \preceq C$ iff $A  \preceq B \multimap C$.
\end{proposition}

\section{Hyper Boolean Algebras}\label{sec3}

In this section we define Hyper Boolean Algebras and provide some basic properties. 

\begin{definition} [Hyper Boolean Algebra]\label{hba}
A {\em hyper Boolean algebra} (or a {\em HBA)}) is a hyperalgebra $\mathsf B=\langle B,\curlywedge,\curlyvee,-, \top,\bot\rangle$ such that $\mathsf B=\langle B,\curlywedge,\curlyvee, \top,\bot\rangle$ is a bounded distributive hyperlattice and for all $x,y\in B$:
\begin{description}
    \item[(HBA 1)  ] If $x\equiv y$ then $- x\equiv- y$.
    \item[(HBA 2) ] $x\in{-}{-} x$.
    \item[(HBA 3)  ] $x\curlyvee- x\equiv\top$.
    \item[(HBA 4)  ] $x\curlywedge- x\equiv\bot$.
\end{description}
\end{definition}

\begin{remark}
It is instructive to compare the axioms of Definition~\ref{hba} with those of Sette hyperalgebras for $C_{\omega}$ ($SHC_{\omega}$) analyzed in~\cite{con:gol:rob:2025} (Definition~16). While both structures validate the Law of Excluded Middle (Axiom (HBA~3) mirrors condition (H1') of $SHC_{\omega}$), the fundamental distinction lies in the Law of Non-Contradiction. HBAs require that $x \curlywedge -x \equiv \bot$ (Axiom (HBA~4)), characterizing the operator $-$ as a classical Boolean negation. In contrast, $SHC_{\omega}$ algebras, being models for a paraconsistent logic, lack this condition regarding their negation operator $\div$. Furthermore, HBAs enforce strict involution ($x \in --x$, which implies $x \equiv --x$ in this context), whereas $SHC_{\omega}$ requires only the weaker condition $\div\div x \preceq x$ (condition (H2')). Thus, HBAs constitute the hyperalgebraic counterpart of Boolean algebras, while $SHC_{\omega}$ correspond to da Costa algebras.
\end{remark}

%

\begin{proposition} \label{prop.HBA}
Let $\mathsf B$ be a HBA. Then, the following holds:
\begin{enumerate}
    \item If $y,z \in -x$ then $y \equiv z$. That is, the set $-x$ is stable.

    \item If $A \subseteq B$ is stable then $A \equiv A \curlywedge \top$ and $A \equiv A \curlyvee \bot$.

    \item If $A \subseteq B$ is stable then $-A$ is stable.

    \item If $A \subseteq B$ is stable then ${-}{-}A \equiv A$. In particular,  $x\equiv{-}{-} x$.

    \item $x \preceq y$ if and only if $x \curlywedge -y \equiv \bot$.

    \item If $x \preceq y$ if and only if  $-x \curlyvee y  \equiv \top$.

    \item If $x \preceq y$ then $-y \preceq -x$.

    \item If $A,A' \subseteq B$ are stable, then: $A \preceq A'$ if and only if $-A' \preceq -A$. In particular, $x \preceq y$ if and only if $-y \preceq -x$.

    \item $-(x\curlywedge y)\equiv- x\curlyvee- y$.

    \item $-(x\curlyvee y)\equiv- x\curlywedge- y$.
\end{enumerate}
\end{proposition}
\begin{proof}
$ $
\begin{enumerate}
    \item Let  $y,z \in -x$. Since $x \equiv x$ then, by (HBA 1), $-x \equiv -x$. This implies that $y \equiv z$.

    \item Assume that $A$  is stable. Let $x \in A$ and $y \in \top$. Then, $x \preceq y$ and so $x \equiv x \curlywedge y$, where $x \curlywedge y \subseteq A \curlywedge \top$. Since $A$ and $\top$ are stable, so is $A \curlywedge \top$ and so  $A \equiv A \curlywedge \top$. The proof that $A \equiv A \curlyvee \bot$ is dual.

    \item Let $y,z \in -A$. Then, $y \in -x$ for some $x \in A$ and $z \in -w$ for some $w \in A$. Since $A$ is stable, $x \equiv w$ and so $-x\equiv -w$, by (HBA~1). From this, $y\equiv z$.

    \item Let $A$ be a stable subset of $B$. By item~(3), $-A$ is stable and then, using~(3) again, ${-}{-}A$ is stable. Let $x \in A$. By  (HBA~2), $x\in{-}{-} x \subseteq {-}{-}A$ and so $x \in {-}{-}A$. This implies that ${-}{-}A \equiv A$, since both sets are stable.

    \item Suppose that $x \preceq y$. Then $x \equiv x \curlywedge y$ and so $x \curlywedge -y \equiv (x \curlywedge y) \curlywedge -y \equiv x \curlywedge (y \curlywedge -y) \equiv x \curlywedge \bot \equiv \bot$, by (HBA~4). Conversely, suppose that $x \curlywedge -y \equiv \bot$. Then, $x \equiv x \curlywedge \top \equiv x \curlywedge (y \curlyvee -y) \equiv (x \curlywedge y) \curlyvee (x \curlywedge -y) \equiv (x \curlywedge y) \curlyvee \bot \equiv  x \curlywedge y$. From this, $x \preceq y$.

    \item The `only if' part is proven as in the previous item by a dual argument,  but now by using (HBA~3). The converse is proven as in the previous item by a dual argument, and by changing $x$ by $y$ and vice versa.

    \item Suppose that $x \preceq y$. By~(5),  $x \curlywedge -y \equiv \bot$. By~(1), the set $-y$ is stable and so, by~(2), (HBA~3) and distributivity, $-y \equiv -y \curlywedge \top \equiv -y \curlywedge (x \curlyvee -x) \equiv (-y \curlywedge x) \curlyvee (-y \curlywedge -x)$. But $-y \curlywedge x=x \curlywedge -y \equiv \bot$, by~(5). Hence, by~(2),  $-y \equiv \bot \curlyvee (-y \curlywedge -x) \equiv -y \curlywedge -x$. That is, $-y \preceq -x$.

    \item Suppose that $A,A' \subseteq B$ are stable. Assume first that $A \preceq A'$, and let $x \in A$ and $y \in A'$. Then, $x \preceq y$ and so $-y \preceq -x$, by~(6). Since $-y \subseteq -A'$ and $-x \subseteq -A$, and both $-A$ and $-A'$ are stable by~(3), then $-A' \preceq -A$. Now, suppose that $-A' \preceq -A$. By~(3), the sets $-A$ and $-A'$ are stable and so, by the first part of this item, ${-}{-}A \preceq {-}{-}A'$. But then, by~(4),   $A \preceq A'$.

    \item Let $x,y \in B$. Since $x \curlywedge y \preceq x$ and $x \curlywedge y \preceq y$ then, by~(7), $-x \preceq -(x \curlywedge y)$ and  $-y \preceq -(x \curlywedge y)$, hence  $-x \curlyvee -y \preceq -(x \curlywedge y)$.
On the other hand, $-x \preceq -x \curlyvee -y$ and $-y \preceq -x \curlyvee -y$, hence $-(-x \curlyvee -y) \preceq {-}{-}x$ and  $-(-x \curlyvee -y) \preceq {-}{-}y$, by~(7), and so  $-(-x \curlyvee -y) \preceq x$ and  $-(-x \curlyvee -y) \preceq y$, by~(4). This means that $-(-x \curlyvee -y) \preceq x \curlywedge y$ and so, by~(7), $-(x \curlywedge y) \preceq {-}{-}(-x \curlyvee -y)$. By~(4), it follows that $-(x \curlywedge y) \preceq -x \curlyvee -y$. This concludes the proof.

    \item The proof is dual to the one given for item~(8).
\end{enumerate}
\end{proof}

\begin{proposition} \label{prop:defimp}
    Every HBA $\mathsf B$ is an IHL with $a\multimap b=- a\curlyvee b$ for every $a,b\in B$.
\end{proposition}
\begin{proof}
    We already know for a hyper lattice $\mathsf L$ that it is implicative (i.e, there exists an implication hyperoperator $\multimap$ satisfying I1-I3) iff $a\multimap b=\mathsf{Max}(\mathsf{R}(a,b))=\mathsf{Max}(\{z\in L:a\curlywedge z\preceq b\})$ for all $a,b\in L$.

    Since
    \begin{align*}
        a\curlywedge(- a\curlyvee b)&\equiv(a\curlywedge- a)\curlyvee(a\curlywedge b)\equiv\bot\curlyvee(a\curlywedge b)\equiv a\curlywedge b\preceq b
    \end{align*}
    we have that $- a\curlyvee b\subseteq\mathsf{R}(a,b)$. Now, if $w\curlywedge a\preceq b$ then
    \begin{align*}
        w&\preceq(- a\curlyvee w)\curlywedge\top\equiv(- a\curlyvee w)\curlywedge(- a\curlyvee a)\equiv- a\curlyvee(w\curlywedge a)\preceq - a\curlyvee b
    \end{align*}
    proving that $- a\curlyvee b\subseteq\mathsf{Max}(\mathsf R(a,b))$. In order to prove that $\mathsf{Max}(\mathsf R(a,b)) \subseteq - a\curlyvee b$, let $w \in \mathsf{Max}(\mathsf R(a,b))$ and let $z \in -a$. Then $z \curlywedge a \equiv \bot \preceq b$ and so $z \in \mathsf R(a,b)$, therefore $z \preceq w$. Since $b \curlywedge a \preceq b$ then $b \in \mathsf R(a,b)$ and so $b \preceq w$. Now, let $x$ such that $z \preceq x$ and $b \preceq x$. Then, $-x \preceq -z \subseteq --a \equiv a$, hence $-x \preceq a$.  From this, $w \curlywedge -x \preceq w \curlywedge a \preceq b$. In turn, from  $b \preceq x$ we infer that $-x \preceq -b$ and so $w \curlywedge -x \preceq -x \preceq - b$. This implies that  $w \curlywedge -x \preceq b \curlywedge -b \equiv \bot$ and so, by Proposition~\ref{prop.HBA}(5), $w\preceq x$. This shows that $w \in z \curlyvee b = -a \curlyvee b$. That is, $\mathsf{Max}(\mathsf R(a,b)) \subseteq - a\curlyvee b$ and so $- a\curlyvee b=\mathsf{Max}(R(a,b))=a\multimap b$.
\end{proof}

\begin{proposition} \label{prop:HBA-IHL}
    Let $\mathsf L$ be a bounded IHL (i.e., a Hyper Heyting Algebra) and define $- x:=x\multimap\bot$, for every $x \in L$. Then $\mathsf L$ is a HBA iff $x\in {-}{-} x$, for every $x \in L$.
\end{proposition}
\begin{proof}
    The `only if' part is immediate from Definition~\ref{hba}. For the `if' part, assume that $\mathsf L$ is a bounded IHL such that $- x:=x\multimap\bot$ satisfies: $x\in {-}{-} x$, for every $x \in L$. It remains to prove that $L$ satisfies conditions (HBA~1), (HBA~3) and (HBA~4) of  Definition~\ref{hba}. Condition  (HBA~4) clearly holds, by definition of $-$. Condition  (HBA~1) holds by stability of $\multimap$: if $x \equiv y$ then $- x=x\multimap\bot \equiv y\multimap\bot = -y$. Finally, let us prove (HBA~3). The proof of items (3) and (4) of Proposition~\ref{prop.HBA} only uses (HBA~1), and so these properties also hold in $\mathsf L$, that is: if $A$ is stable so is $-A$, and ${-}{-}A\equiv A$. Now, let $x \in L$. Then, $-(x \curlyvee -x) \curlywedge x\preceq -(x \curlyvee -x) \curlywedge (x \curlyvee -x) \preceq \bot$ and so $-(x \curlyvee -x) \preceq x \multimap \bot = -x$, by Proposition~\ref{prop:IHL1}(2). From this, $-(x \curlyvee -x) \preceq x \curlyvee -x$ and so $-(x \curlyvee -x) \preceq -(x \curlyvee -x) \curlywedge (x \curlyvee -x)\preceq \bot$. Since $x \curlyvee -x$ is stable, this implies that $x \curlyvee -x \equiv {-}{-}(x \curlyvee -x) = \top$, by Proposition~\ref{prop:IHL1}(1).
\end{proof}

\begin{example}
Let $\mathcal{V} = \{p_0,p_1,p_2,\ldots\}$ be  denumerable set of propositional variables, and let $For(\Sigma_C)$ be the free algebra of propositional formulas generated by $\mathcal{V}$ over the signature $\Sigma_C=\{\vee,\wedge, \to, \neg\}$.\footnote{When the arity of the connectives is obvious from the notation, as in the present case, (finite) signatures will we represented simply as a set of symbols.} Define the following relation in $For(\Sigma_C)$: $\alpha \preceq \beta$ iff $\alpha \vdash_{CPL} \beta$ (iff $\vdash_{CPL} \alpha \to \beta$), where $\vdash_{CPL}$ is the (Tarskian and finitary) consequence relation of classical logic CPL presented over the signature $\Sigma_C$. As discussed in~\cite[Example~1(2.)]{con:gol:rob:2025}, $\langle For(\Sigma_C), \preceq\rangle$ is an hyperlattice in which $\alpha \equiv \beta$ iff $\vdash_{CPL}(\alpha \leftrightarrow \beta)$,\footnote{As usual, $\alpha \leftrightarrow \beta$ is an abbreviation for $(\alpha \to \beta) \land (\beta \to \alpha)$.} and infimoids and supremoids are given, respectively, by
$$\alpha \curlywedge \beta=\{ \varphi \ : \ \vdash_{CPL} \varphi \leftrightarrow (\alpha \land\beta)\} \ \ \mbox{ and } \ \ \alpha \curlyvee \beta=\{ \varphi \ : \ \vdash_{CPL} \varphi \leftrightarrow (\alpha \lor\beta)\}.$$
Moreover, it is a HBA in which $-\alpha=\{ \varphi \ : \  \ \vdash_{CPL} \varphi \leftrightarrow \neg \alpha\}$, $\top=\{ \varphi \ : \ \  \vdash_{CPL} \varphi\}$  and  $\bot=\{ \varphi \ : \ \ \vdash_{CPL} \neg\varphi\}$. From this, and by Proposition~\ref{prop:defimp}, it is easy to show that
$$\alpha \multimap \beta=\{ \varphi \ : \ \ \vdash_{CPL} \varphi \leftrightarrow (\alpha \to\beta)\}=\{ \varphi \ : \ \ \vdash_{CPL} \varphi \leftrightarrow (\neg\alpha \vee\beta)\}.$$
\end{example}

\begin{definition}
    Let $A_1,A_2$ be HBA's. A function $f:A_1\rightarrow A_2$ is a \textbf{HBA-morphism} if it satisfies the following, for every $x,y \in A_1$:\\[1mm]
    $\begin{array}{ll}
        (1) \ f(x \curlywedge y) \subseteq f(x) \curlywedge f(y); & (2) \ f(x \curlyvee y) \subseteq f(x) \curlyvee f(y); \\[1mm]
         (3) \ f(\top) \subseteq \top; & (4) \ f(\bot) \subseteq \bot; \\[1mm]
        (5) \ f(-x)\subseteq-f(x).
    \end{array}$
\end{definition}

\section{The Logic mbC}\label{sec4}

The logic {mbC}, extensively analyzed in~\cite[Chapter 2]{carnielli2016paraconsistent}, serves as the cornerstone of the Logics of Formal Inconsistency (LFIs). It is defined as the minimal extension of positive classical logic ($\mathsf{CPL}^{+}$) containing a paraconsistent negation $\neg$ and a consistency operator $\circ$ that satisfies the conditions of an LFI. The axiom (bc1) (see Definition~\ref{def:mbC}), known as the \emph{gentle explosion law}, ensures that the Principle of Explosion applies only to formulas explicitly marked as consistent. As shown in~\cite{carnielli2016paraconsistent}, classical logic can be fully recovered within {mbC} (for instance, via a definable strong negation $\sim \alpha \equiv \alpha \to (\neg \alpha \land \circ \alpha)$), which justifies our investigation into its algebraic counterpart---the Hyper Boolean Algebras---as a natural generalization of Boolean Algebras.

\begin{definition} [Hilbert Calculus for  $\mathsf{CPL}^{+}$]\label{hilbert-calculus}
The Hilbert calculus for $\mathsf{CPL}^{+}$ is defined over the signature $\Sigma_+= \{\land,\vee,\to\}$  as follows:
\\[2mm]
{\bf Axiom schemas:}
\begin{multicols}{2}
    \begin{description}
    \item[(AX1)] $\alpha  \to (\beta  \to \alpha )$
    \item[(AX2)] $(\alpha  \to (\beta  \to \gamma )) \to ((\alpha  \to \beta ) \to (\alpha  \to \gamma ))$
    \item[(AX3)] $\alpha  \to (\beta  \to (\alpha  \land \beta ))$
    \item[(AX4)] $(\alpha  \land \beta ) \to \alpha$
    \item[(AX5)] $(\alpha  \land \beta ) \to \beta$
    \item[(AX6)] $\alpha  \to (\alpha  \lor \beta )$
    \item[(AX7)] $\beta  \to (\alpha  \lor \beta )$
    \item[(AX8)] $(\alpha  \to \gamma ) \to ((\beta  \to \gamma ) \to	((\alpha  \lor \beta ) \to \gamma ))$
    \item[(AX9)] $\alpha  \lor (\alpha \to \beta)$
\end{description}
\end{multicols}

\textbf{Inference rule:}
\begin{description}
    \item[(MP)] $\inferrule{\alpha\quad\alpha\rightarrow\beta}{\beta}$ 
\end{description}
\end{definition}

It is worth noting that $\mathsf{CPL}^{+}$ is semantically characterized by the class of classical implicative lattices.

\begin{definition}[Logic {mbC} \cite{car:con:16}] \label{def:mbC}
    The logic {mbC} is given by the Hilbert calculus in the signature $\Sigma = \{\land,\vee,\to,\neg, \circ\}$ obtained by adding to $\mathsf{CPL}^{+}$ the following axioms:\\[1mm]
     $\begin{array}{ll}
       \mbox{\bf(EM)} \ \alpha\lor \neg \alpha & \hspace{1cm} \mbox{\bf(bc1)} \ \circ \alpha\to (\alpha \to (\neg \alpha \to \beta))
    \end{array}$
\end{definition}

\section{Hyper mbC Algebras and Hyper Swap structures for mbC}\label{sec5}

The concept of \emph{swap structures} for LFIs was originally introduced in~\cite[Chapter 6]{carnielli2016paraconsistent} as a family of multialgebras defined over standard Boolean Algebras, providing a unified non-deterministic semantic framework for these logics. In this section, we generalize this construction by defining swap structures over \emph{Hyper Boolean Algebras} (HBAs). This approach mirrors the strategy successfully applied in~\cite{con:gol:rob:2025} for the logic $C_\omega$, where swap structures were defined over implicative hyperlattices. By replacing the underlying Boolean Algebra with a HBA, we derive the notion of \emph{Hyper Swap Structures} for {mbC}, denoted as $S^{mbC}(B)$. These structures naturally induce the class of \emph{Hyper mbC Algebras} (HmbCAs), which validate the Excluded Middle and the Gentle Explosion Law within a hyperalgebraic context, thereby extending the standard valuation semantics.

\begin{definition}\label{hyper-mbc-algebra}
    A {\em hyperalgebra for mbC} (or {\em hyper  mbC algebra}, or simply a HmbCA) is a hyperalgebra $\mathsf A=\langle A,\curlywedge,\curlyvee,-,\bot,\top,\neg,\circ\rangle$ such that the reduct~\mbox{$\langle A,\curlywedge,\curlyvee,\bot,\top\rangle$}  is an HBA and the hyperoperators $\neg$ and $\circ$  satisfy the following properties, for every $x, y, w \in A$:

\begin{description}
    \item[(EM)] $x\curlyvee y\equiv\top$ for all $y \in \neg x$;
    \item[(BC)] $(x\curlywedge y)\curlywedge z\equiv\bot$ for all $y \in \neg x$ and all $z\in\circ x$.
\end{description}

\noindent It is immediate to see that conditions (EM) and (BC) can be written in a concise way as follows:

\begin{description}
    \item[(EM')] $x \curlyvee \neg x \equiv \top$;
    \item[(BC')] $(x\curlywedge\neg x)\curlywedge\circ x \equiv \bot$.
\end{description}

\noindent
for every $x \in A$. 
\end{definition}
\begin{definition}
    A {\em non-deterministic matrix} (or an {\em Nmatrix}) over $\Theta$ is a pair $\mathcal{M}=\langle \mathsf A, D
    \rangle$  such that $\mathsf A$  is a $\Theta$-hyperalgebra and $\emptyset \neq D \subseteq A$. A {\em valuation} over $\mathcal{M}$ is a function $v:For(\Theta) \to A$ such that $v(\#(\varphi_1,\ldots,\varphi_n)) \in \mathcal{O}(\#)(v(\varphi_1),\ldots,v(\varphi_n))$ for $\# \in \Theta_n$ and $\varphi_1,\ldots,\varphi_n \in For(\Theta)$.
The logic generated by an Nmatrix $\mathcal{M}$ is defined as follows: $\Gamma \models_{\mathcal{M}}\varphi$ iff, for every valuation $v$ over $\mathcal{M}$, if  $v(\gamma)\in D$ for every $\gamma \in \Gamma$ then $v(\varphi)\in D$.
\end{definition}

\begin{definition} [HmbCA semantics] \label{def-sem-HmbCA} Let $\mathsf A$ be a HmbCA, and let $\Gamma \cup \{\varphi\}$ be a set of formulas over $\Sigma$.
\begin{enumerate}
    \item The Nmatrix associated with $\mathsf A$ is $\mathcal M^{\mathsf{HmbCA}}_{\mathsf A}=\langle \mathsf A, \top\rangle$.

    \item  $\varphi$ is a semantical consequence of $\Gamma$ w.r.t. $\mathsf A$ if $\Gamma \models_{\mathcal M^{\mathsf{HmbCA}}_{\mathsf A}} \varphi$.

    \item Let $\mathsf{HmbCA}$ be the class of HmbCAs. Then, $\varphi$ is a semantical consequence of $\Gamma$ w.r.t.  HmbCAs, denoted by $\Gamma \models_{\mathsf{HmbCA}} \varphi$, if $\Gamma \models_{\mathcal M^{\mathsf{HmbCA}}_{\mathsf A}} \varphi$ for every $\mathsf A \in \mathsf{HmbCA}$.
\end{enumerate}
\end{definition}

\begin{definition} [Hyper Swap structures for mbC] \label{def-hyper-swap-mbc}
Let $\mathsf B=\langle B, \curlywedge,\curlyvee,-, \top,\bot\rangle$ be a HBA. Let 
$$S^{\mathsf{mbC}}_{\mathsf B}=\{z \in B \times B\times B \ : \ z_1 \curlyvee z_2\equiv\top\mbox{ and } z_1\curlywedge z_2\curlywedge z_3\equiv\bot\}.$$
The {\em hyper swap structure} for $mbC$ over $\mathsf B$ is the hyperalgebra $\mathsf S^{\mathsf{mbC}}(\mathsf B)=\langle S^{\mathsf{mbC}}_{\mathsf B}, \curlywedge,\curlyvee,\multimap,\neg,\circ\rangle$ over the signature $\Sigma$ such that the hyperoperators are defined as follows:\footnote{By simplicity, the hyperoperators and the induced preordering in $\mathsf S^{\mathsf{mbC}}(\mathsf B)$ will be denoted by using the same symbols than in $\mathsf B$. The context will avoid any confusion.}
\begin{align*}
    z\curlywedge w&:=\{u\in S^{\mathsf{mbC}}_{\mathsf B} \ : \ u_1\in z_1\curlywedge w_1 \}\\
    z\curlyvee w&:=\{u\in S^{\mathsf{mbC}}_{\mathsf B} \ : \ u_1\in z_1\curlyvee w_1 \}\\
    z\multimap w&:=\{u\in S^{\mathsf{mbC}}_{\mathsf B} \ : \ u_1\in z_1\multimap w_1 \}\\
    \neg z&:=\{u\in S^{\mathsf{mbC}}_{\mathsf B} \ : \ u_1=z_2\}\\
    \circ z&:=\{u\in S^{\mathsf{mbC}}_{\mathsf B} \ : \ u_1=z_3\}.
\end{align*}
 We may denote $S^{\mathsf{mbC}}(B)$ simply by $S^{\mathsf{mbC}}_{\mathsf B}$.
\end{definition}

The hyperoperations in  $S^{\mathsf{mbC}}_{\mathsf B}$  can be described in a more succinct way as follows:
\begin{align*}
    z\curlywedge w&=(z_1\curlywedge w_1, \, \_,\, \_)\\
    z\curlyvee w&=(z_1\curlyvee w_1, \, \_,\, \_)\\
    z\multimap w&=(z_1\multimap w_1, \, \_,\, \_)\\
    \neg z&=(z_2, \, \_,\, \_)\\
    \circ z&=(z_3, \, \_,\, \_)
\end{align*}


Following the usual definitions for swap structures, each hyper swap structure can be naturally associated with an Nmatrix:

\begin{definition} \label{def-Nmatrix-mbC}
 Let $\mathsf A$ be a HBA. The Nmatrix associated with $\mathsf S^{\mathsf{mbC}}(\mathsf A)$ is $\mathcal{M}^{\mathsf{mbC}}(\mathsf A)=\langle \mathsf S^{\mathsf{mbC}}(\mathsf A),D^{\mathsf{mbC}}(\mathsf A) \rangle$ where the set of designated truth-values is $D^{\mathsf{mbC}}(\mathsf A)=\{z \in \mathsf S^{\mathsf{mbC}}_{\mathsf A} \ : \ z_1\in\top\}$.
\end{definition}

\begin{proposition} \label{swap-SHmbC}
Let $\mathsf A$ be  a HBA, and let $\mathsf S^{\mathsf{mbC}}(\mathsf A)$ be the hyper swap structure for mbC over $\mathsf A$. Then:
\begin{enumerate}
    \item The relation $z \preceq w$  in  $\mathsf S^{\mathsf{mbC}}(\mathsf A)$  iff  $z_1 \preceq w_1$  in $\mathsf A$
    defines a  preorder such that $\mathsf S^{\mathsf{mbC}}(\mathsf A)$ is an hyperlattice where, for every $z,w \in  S^{\mathsf{mbC}}(\mathsf A)$, $z\curlywedge w$ and $z\curlyvee w$ are the infimoid and the supremoid of $z$ and $w$, respectively.
    Moreover, $z \equiv w$ in  $\mathsf S^{\mathsf{mbC}}(\mathsf A)$ iff $z_1 \equiv w_1$ in $\mathsf A$. 
    \item $\mathsf S^{\mathsf{mbC}}(\mathsf A)$ is a HmbCA.
    Moreover, $D^{\mathsf{mbC}}(\mathsf A)=\top$.
    \item $\mathcal{M}^{\mathsf{mbC}}(\mathsf A)= \mathcal M_{\mathsf S^{\mathsf{mbC}}(\mathsf A)}$.
\end{enumerate}
\end{proposition}
\begin{proof}
$ $
\begin{enumerate}
    \item The verification that $\preceq$ is a preorder in $\mathsf S^{\mathsf{mbC}}(\mathsf A)$ is straightforward.
    
{\bf Fact 1:} If $a \in A$, $b \in \top$ and $c \in \bot$ then $(a,b,c) \in \mathsf S^{\mathsf{mbC}}(\mathsf A)$.

Indeed, if $b \in \top$, $c \in \bot$ and $x \in A$ then  $x \preceq b \preceq a \curlyvee b$. This means that $a \curlyvee b \equiv \top$. In turn, $a \curlywedge b \curlywedge c \preceq c \preceq x$, and so $a \curlywedge b \curlywedge c \equiv \bot$. This means that $(a,b,c) \in \mathsf S^{\mathsf{mbC}}(\mathsf A)$. This proves  {\bf Fact~1}.

As a direct consequence of {\bf Fact 1}, and given that $z_1 \curlywedge w_1$ and $z_1 \curlyvee w_1$ are non empty sets, it follows that $z \curlywedge w$ and $z \curlyvee w$ are non empty sets for every $z,w \in \mathsf S^{\mathsf{mbC}}(\mathsf A)$.
Now, let us prove that $z\curlyvee w$ is the supremoid in  $\mathsf S^{\mathsf{mbC}}(\mathsf A)$ of $z$ and $w$. Thus, let $u \in z\curlyvee w$. By definition, $u_1 \in z_1 \curlyvee w_1$ and so $z_1,w_1 \preceq u_1$. Then, $z,w \preceq u$ and so $u \in \mathsf{Ub}(\{z,w\})$. Let $x \in \mathsf{Ub}(\{z,w\})$. Then, $z,w \preceq x$ and so $z_1,w_1 \preceq x_1$, therefore $z_1 \curlyvee w_1 \preceq x_1$. From this, $u_1 \preceq x_1$, hence $u \preceq x$. This shows that $u \in \mathsf{Min}(\mathsf{Ub}(\{z,w\}))$, that is, $z\curlyvee w \subseteq \mathsf{Min}(\mathsf{Ub}(\{z,w\}))$. Now, let  $u \in \mathsf{Min}(\mathsf{Ub}(\{z,w\}))$. Since $z,w \preceq u$ then $z_1,w_1 \preceq u_1$, that is, $u_1 \in \mathsf{Ub}(\{z_1,w_1\})$. Let $a \in  \mathsf{Ub}(\{z_1,w_1\})$. By {\bf Fact~1}, $(a,b,c) \in  \mathsf S^{\mathsf{mbC}}(\mathsf A)$ for any  $b \in \top$ and $c \in \bot$, and $z,w \preceq (a,b,c)$. Thus, $u \preceq (a,b,c)$ which implies that $u_1 \preceq a$. Then, $u_1 \in \mathsf{Min}(\mathsf{Ub}(\{z_1,w_1\})) = z_1 \curlyvee w_1$. That is, $u \in  z \curlyvee w$ and so $z\curlyvee w = \mathsf{Min}(\mathsf{Ub}(\{z,w\}))$. The proof that $z\curlywedge w = \mathsf{Max}(\mathsf{Lb}(\{z,w\}))$ is analogous. This shows that  $\mathsf S^{\mathsf{mbC}}(\mathsf A)$ is an hyperlattice where $z \equiv w$ in $\mathsf S^{\mathsf{mbC}}(\mathsf A)$ iff $z_1 \equiv w_1$ in $\mathsf A$.

    \item Note that $\mathsf{Max}(\mathsf S^{\mathsf{mbC}}(\mathsf A))= \{z \in \mathsf S^{\mathsf{mbC}}_{\mathsf A} \ : \ z_1\in\top\}$. By {\bf Fact 1}, this set is non empty. Analogously, $\mathsf{Min}(\mathsf S^{\mathsf{mbC}}(\mathsf A))= \{z \in \mathsf S^{\mathsf{mbC}}_{\mathsf A} \ : \ z_1\in\bot\}$ is non empty. This shows that $\mathsf S^{\mathsf{mbC}}(\mathsf A)$ is bounded. Without risk of confusion, the sets $\mathsf{Min}(\mathsf S^{\mathsf{mbC}}(\mathsf A))$ and $\mathsf{Max}(\mathsf S^{\mathsf{mbC}}(\mathsf A))$ will be denoted, respectively, by $\bot$ and $\top$. Let $-z:=z \multimap \bot$, for any $z \in S^{\mathsf{mbC}}_{\mathsf B}$.    
    The validity of Axioms (HBA~1), (HBA~3) and (HBA~4) of Definition~\ref{hba} for $\mathsf S^{\mathsf{mbC}}(\mathsf A)$ follows from the validity of these axioms in $\mathsf A$ and the fact that if $z,w\in \mathsf S^{\mathsf{mbC}}(\mathsf A)$, then $z \preceq w$ ($z\equiv w$) iff $z_1\preceq w_1$ ($z_1\equiv w_1$). For (HBA~2), let $z\in S^{\mathsf{mbC}}(\mathsf A)$. Since $-a=a \multimap \bot$ in $\mathsf A$ (for every $a \in A$) then $- z=\{w\in S^{\mathsf{mbC}}(\mathsf A) \ : \ w_1 \in -z_1\}$. Thus $z\in {-}{-} z=\{w\in S^{\mathsf{mbC}}(\mathsf A) \ : \ w_1 \in {-}{-}z_1\}$, since  $\mathsf A$ satisfies (HBA~2). Then, $\mathsf S^{\mathsf{mbC}}(\mathsf A)$ is a HBA which is also a HmbCA by construction (the verification of the validity of EM and BC of Definition~\ref{hyper-mbc-algebra} is similar). Moreover, we have
    \begin{align*}
        D^{\mathsf{mbC}}(\mathsf A)&=\{z \in \mathsf S^{\mathsf{mbC}}_{\mathsf A} \ : \ z_1\in\top\}=\mathsf{Max}(\mathsf S^{\mathsf{mbC}}(\mathsf A))=\top.
    \end{align*}

    \item It follows by the fact that $D^{\mathsf{mbC}}(\mathsf A)= \mathsf{Max}(\mathsf S^{\mathsf{mbC}}(\mathsf A))$, and from the definitions.
\end{enumerate}
\end{proof}

\begin{definition} [Hyper Swap structures semantics for  mbC] \label{def-sem-swap-mbC} Let $\Gamma \cup \{\varphi\}$ be a set of formulas over mbC. Then, $\varphi$ is a semantical consequence of $\Gamma$ w.r.t. hyper swap structures, denoted by $\Gamma \models_{\mathsf{mbC}}^{\mathsf{HSW}} \varphi$, whenever $\Gamma \models_{\mathcal{M}^{\mathsf{mbC}}(\mathsf A)} \varphi$ for every hyper Boolean algebra $\mathsf A$.
\end{definition}

\noindent In order to prove  soundness and completeness of $mbC$ w.r.t. hyper swap structures semantics, we recall here some well-known  notions and results concerning (Tarskian) logics.

Given a Tarskian and finitary logic {\bf L} and a set of formulas $\Delta \cup \{ \varphi\}$ of {\bf L}, the set $\Delta$ is said to be {\em $\varphi$-saturated in} {\bf L} if the following holds:~(i)~$\Delta \nvdash_{\bf L} \varphi$; and~(ii)~if $\psi \notin \Delta$ then $\Delta,\psi \vdash_{\bf L}\varphi$.

It follows immediately that any  $\varphi$-saturated set in a  Tarskian logic is deductively closed, i.e.: $\psi \in \Delta$ iff $\Delta \vdash_{\bf L} \psi$.

By a classical result proven by Lindenbaum and \L o\'s,  if $\Gamma \cup \{ \varphi\}$ is a set of formulas of a Tarskian and finitary logic {\bf L} such that $\Gamma \nvdash_{\bf L} \varphi$, then there exists a $\varphi$-saturated set $\Delta$  such that $\Gamma \subseteq \Delta$.\footnote{For a proof of this result see, for instance, \cite[Theorem~22.2]{wojcicki1984lectures} or~\cite[Theorem~2.2.6]{carnielli2016paraconsistent}.} Since mbC is a Tarskian and finitary logic, Lindenbaum-\L o\'s Theorem holds for it.

\begin{theorem} [Soundness and completeness of mbC w.r.t. hyperstructures semantics] \label{Sound-comple-mbC} \ \\
Let $\Gamma \cup \{\varphi\}$ be a set of formulas over $\Sigma$. The following assertions are equivalent:
\begin{enumerate}
    \item $\Gamma \vdash_{\mathsf{mbC}} \varphi$;
    \item $\Gamma \models_{\mathsf{HmbCA}} \varphi$;
    \item $\Gamma \models_{\mathsf{mbC}}^{\mathsf{HSW}} \varphi$.
\end{enumerate}
\end{theorem}
\begin{proof}
$(1) \Rightarrow (2)$ (Soundness of   mbC w.r.t. HmbCAs). The verification of validity of axioms (AX1)-(Ax8) and axiom EM was proven in~\cite[Theorem~1]{con:gol:rob:2025}. After this, the validity of axioms (AX9) and BC is obtained after a straightforward calculation involving the axioms of hyper mbC algebra and HBA. \\[1mm]
$(2) \Rightarrow (3)$. It follows by Proposition~\ref{swap-SHmbC}, items~(2) and~(3).\\[1mm]
$(3) \Rightarrow (1)$ (Completeness of mbC w.r.t. hyper swap structures semantics). Suppose that $\Gamma \nvdash _{\mathsf{mbC}} \varphi$. Then, by the Lindenbaum-\L o\'s result mentioned above, there exists a $\varphi$-saturated set $\Delta$ in mbC such that $\Gamma \subseteq \Delta$. Now, define  a relation $\preceq_\Delta$ over $For(\Sigma)$ as follows: $\alpha \preceq_\Delta \beta$ iff $\Delta \vdash _{\mathsf{mbC}} \alpha \to \beta$. We have that $\preceq_\Delta$ is a preorder, given that mbC contains positive classical logic $\mathsf{CPL^+}$, hence $\vdash_{\mathsf{mbC}} \alpha \to \alpha$, and  $\alpha \to \beta, \, \beta \to \gamma \vdash_{\mathsf{mbC}} \alpha \to \gamma$. Observe that, with this preorder, $\alpha \equiv \beta$ iff  $\Delta \vdash _{\mathsf{mbC}} \alpha \leftrightarrow \beta$, where $\alpha \leftrightarrow \beta$ is an abbreviation for  $(\alpha \to \beta) \land (\beta \to \alpha)$. Using again the fact that mbC is an axiomatic extension of $\mathsf{CPL^+}$, it is straightforward to prove that $\langle For(\Sigma), \preceq_\Delta \rangle$ is an hyperlattice such that $\alpha \curlywedge \beta =\{\gamma \ : \ \Delta \vdash _{\mathsf{mbC}} \gamma \leftrightarrow (\alpha \land\beta)\}$, and  $\alpha \curlyvee \beta =\{\gamma \ : \ \Delta \vdash _{\mathsf{mbC}} \gamma \leftrightarrow (\alpha \vee\beta)\}$. Moreover, it is a HBA (that we will call $\mathsf{A}_\Delta$)  such that ${-}\alpha=\{\gamma \ : \ \Delta \vdash _{\mathsf{mbC}} \gamma \leftrightarrow {-}\alpha\}$ and $\alpha \multimap \beta =\{\gamma \ : \ \Delta \vdash _{\mathsf{mbC}} \gamma \leftrightarrow (\alpha \to\beta)\}$. Observe that $\top=\{\gamma \ : \ \Delta \vdash _{\mathsf{mbC}} \gamma\}=\Delta$ and $\bot=\{\gamma \ : \ \Delta \vdash _{\mathsf{mbC}}\gamma \leftrightarrow \bot\}=\bot_{\mathsf{mbC}}$.

Let $\mathsf S^{\mathsf{mbC}}(\mathsf{A}_\Delta)$ be the hyper swap structure for mbC over $\mathsf{A}_\Delta$, with domain $S^{\mathsf{mbC}}_{\mathsf{A}_\Delta}$, as in Definition~\ref{def-hyper-swap-mbc}, and let $\mathcal{M}^{\mathsf{mbC}}(\mathsf{A}_\Delta)$ be the associated Nmatrix (see Definition~\ref{def-Nmatrix-mbC}). Notice that 
$$S^{\mathsf{mbC}}(\mathsf{A}_\Delta)=\{(\alpha,\beta,\gamma) \ : \ \Delta \vdash_{\mathsf{mbC}} \alpha \vee \beta\mbox{ and }\Delta \vdash_{\mathsf{mbC}}\alpha\wedge\beta\wedge\gamma\rightarrow\bot_{\mathsf{mbC}}\}$$
and 
$$D^{\mathsf{mbC}}(\mathsf{A}_\Delta)=\{(\alpha,\beta,\gamma) \ : \ \Delta \vdash_{\mathsf{mbC}}\alpha\}.$$
Let $v_\Delta:For(\Sigma) \to S^{\mathsf{mbC}}(\mathsf{A}_\Delta)$ be the canonical map given by $v_\Delta(\alpha)= (\alpha,\neg \alpha,\circ\alpha)$, for every $\alpha$. We have that $v_\Delta$ is a valuation over $\mathcal{M}^{\mathsf{mbC}}(\mathsf{A}_\Delta)$. Indeed, \\

$
\begin{array}{lll}
     v_\Delta(\alpha\land \beta)&=& (\alpha\land\beta,\neg (\alpha\land \beta),\circ(\alpha\land\beta)) \\[1mm]
     &\in& \{(\delta_1,\delta_2,\delta_3) \in S^{\mathsf{mbC}}(\mathsf{A}_\Delta) \ : \ \delta_1 \in \alpha \curlywedge \beta\}\\[1mm] 
     &=& v_\Delta(\alpha) \curlywedge v_\Delta(\beta),\\[1mm]
\end{array}$


\noindent
given that $\alpha \land \beta \in \alpha \curlywedge \beta=\{\gamma \ : \ \Delta \vdash _{\mathsf{mbC}} \gamma \leftrightarrow (\alpha \land\beta)\}$. Analogously we prove that 
$$v_\Delta(\alpha\vee \beta) \in v_\Delta(\alpha) \curlyvee v_\Delta(\beta) \ \ \mbox{ and } \ \ v_\Delta(\alpha\to \beta) \in v_\Delta(\alpha) \multimap v_\Delta(\beta).$$
By axioms (EM) and (BF) we have that $\Delta\vdash_{\mathsf{mbC}}\alpha\vee\neg\alpha$ and $\Delta\vdash_{\mathsf{mbC}}\alpha\wedge\neg\alpha\wedge\circ\alpha\rightarrow\bot_{\mathsf{mbC}}$, which provide
\begin{align*}
    v_\Delta(\neg\alpha)&= (\neg\alpha,\neg\neg \alpha,\circ\neg\alpha) \in \{(\delta_1,\delta_2,\delta_3)\in S^{\mathsf{mbC}}(\mathsf{A}_\Delta) \ : \ \delta_1 = \neg\alpha\}
    = \neg v_\Delta(\alpha).
\end{align*}
Similarly we get $v_\Delta(\circ\alpha)\in\circ v_\Delta(\alpha)$. Moreover, $v_\Delta(\alpha)\in D^{\mathsf{mbC}}(\mathsf{A}_\Delta)$ iff $\Delta \vdash_{\mathsf{mbC}} \alpha$. From this, $v_\Delta(\alpha)\in D^{\mathsf{mbC}}(\mathsf{A}_\Delta)$ for every $\alpha \in \Gamma$, while $v_\Delta(\varphi)\notin D^{\mathsf{mbC}}(\mathsf{A}_\Delta)$, given that $\Delta \nvdash _{\mathsf{mbC}} \varphi$. This shows that $\Gamma \not\models_{\mathcal{M}^{\mathsf{mbC}}(\mathsf{A}_\Delta)} \varphi$ and so $\Gamma \not\models_{\mathsf{mbC}}^{\mathsf{HSW}} \varphi$.

This completes the proof.
\end{proof}

\begin{example} 
Consider the following set $B=\{0,a,b\}$ with the relation $\preceq$ on $B$ defined by the following prescriptions:
\begin{align*}
    0\preceq0,&\qquad a\preceq a,\qquad b\preceq b,\qquad 0\preceq a,\qquad0\preceq b,\qquad a\preceq b,\qquad b\preceq a.
\end{align*}
    
    In other words, $B$ is one of the HBAs such that $(B/\equiv) \cong\bm2$, where $\bm2=\{\bot,\top\}$ is the 2-element Boolean algebra. The hyperoperators on $B$ are given by:
    \begin{center}
\begin{tabular}{| c || c | c | c | }
\hline 
$\curlywedge$ & $0$ & $a$ & $b$ \\ \hline \hline
$0$ & $0$ & $0$ & $0$\\ \hline
$a$ & $0$ & $\{a,b\}$ & $\{a,b\}$\\ \hline
$b$ & $0$ & $\{a,b\}$ & $\{a,b\}$\\ \hline
\end{tabular} \hspace{5mm}
\begin{tabular}{| c || c | c | c | }
\hline
$\curlyvee$ & $0$ & $a$ & $b$ \\ \hline \hline
$0$ & $0$ & $\{a,b\}$ & $\{a,b\}$\\ \hline
$a$ & $\{a,b\}$ & $\{a,b\}$ & $\{a,b\}$\\ \hline
$b$ & $\{a,b\}$ & $\{a,b\}$ & $\{a,b\}$\\ \hline
\end{tabular} 

\ \\[2mm]

\begin{tabular}{| c || c | c | c | }
\hline
$\multimap$ & $0$ & $a$ & $b$ \\ \hline \hline
$0$ & $\{a,b\}$ & $\{a,b\}$ & $\{a,b\}$\\ \hline
$a$ & $0$ & $\{a,b\}$ & $\{a,b\}$\\ \hline
$b$ & $0$ & $\{a,b\}$ & $\{a,b\}$\\ \hline
\end{tabular} \hspace{5mm}
\begin{tabular}{| c ||  c |}
\hline
 & $-$  \\ \hline\hline
$0$ & $\{a,b\}$ \\ \hline
$a$ & $0$\\ \hline
$b$ & $0$ \\ \hline
\end{tabular}
\end{center}

    The hyper swap structure for mbC over $L$ is given by
    \begin{align*}
        \mathcal S^{\mathsf{mbC}}(B)&=\{(a,0,0),(a,0,b),(b,0,a),(b,0,b),
        (a,a,0),(a,b,0),(b,a,0),(b,b,0),\\
        &(0,a,a),(0,a,b),(0,b,a),(0,b,b),
        (a,0,0),(b,0,0),(0,a,0),(0,b,0)\}.
    \end{align*}
\end{example}

\section{Kalman Functors for mbC}\label{sec6}

Having established the algebraic definitions, we now address the categorical relationship between Hyper Boolean Algebras and the hyperalgebraic semantics for {mbC}. As discussed in the Introduction, a direct adjunction between the category of standard Boolean Algebras and the category of hyper swap structures is obstructed by the non-deterministic nature of the latter. To resolve this, we adapt the categorical framework developed in~\cite{con:gol:rob:2026} for the logic $C_\omega$. By restricting our attention to the category of \emph{Enriched Hyper mbC Algebras} (EHmbCA), we prove that the functor $S^{\mathsf{mbC}}$ induces an equivalence of categories between HBAs and EHmbCAs. This result generalizes the algebraic analysis of swap structures initiated in~\cite[Chapter 6]{carnielli2016paraconsistent} to the categorical level, formally legitimizing HBAs as the algebraic counterpart of {mbC} in the hyperalgebraic setting.

\begin{definition}
    The category \textbf{HmbCA} is the one where the objects are HmbCAs and the morphisms are just the usual morphisms of hyperalgebras. In other words, given $\mathsf A_1,\mathsf A_2\in\mbox{HmbCA}$, a function $f:A_1\rightarrow A_2$ is  a HmbCA-morphism from $\mathsf A_1$ to $\mathsf A_2$ if it is a HBA-morphism such that, for all $x,y\in A_1$, the following holds:
    \begin{enumerate}
        \item $y\in\neg x$ implies $f(y)\in\neg f(x)$;
        \item $y\in\circ x$ implies $f(y)\in\circ f(x)$.
    \end{enumerate}
\end{definition}

\begin{theorem}
    The hyper swap structure construction provides a Kalman functor $\mathsf S^{\mathsf{mbC}}:\textbf{HBA}\rightarrow \textbf{HmbCA}$.
\end{theorem}
\begin{proof}
    We only need to extend $\mathsf S^{\mathsf{mbC}}$ to  morphisms. Let $f:\mathsf A_1\rightarrow \mathsf A_2$ be a \textbf{HBA}-morphism. Define $\mathsf S^{\mathsf{mbC}}(f):S^{\mathsf{mbC}}(\mathsf A_1)\rightarrow S^{\mathsf{mbC}}(\mathsf A_2)$ by the rule $\mathsf S^{\mathsf{mbC}}(f)(z_1,z_2,z_3):=(f(z_1),f(z_2),f(z_3))$. Since $f$ is a morphism,  $z_1\curlyvee z_2\equiv\top$ and $z_1\curlywedge z_2\curlywedge z_3\equiv\bot$, we get $f(z_1)\curlyvee f(z_2)\equiv\top$ and $f(z_1)\curlywedge f(z_2)\curlywedge f(z_3)\equiv\bot$, proving that $\mathsf{S}^{\mathsf{mbC}}(f)$ is well defined. After that, the fact that $\mathsf S^{\mathsf{mbC}}(1_{\mathsf A_1})=1_{\mathsf S^{\mathsf{mbC}}(\mathsf A_1)}$ and $\mathsf S^{\mathsf{mbC}}(f\circ g)=\mathsf S^{\mathsf{mbC}}(f)\circ \mathsf S^{\mathsf{mbC}}(g)$ is a straightforward calculation involving the above definitions.
\end{proof}

\begin{definition}[Enriched Hyper mbC Algebras] \label{enriched-mbc} \ \\
    Let $\mathsf A=\langle A,\curlywedge,\curlyvee,-,\bot,\top,\neg,\circ\rangle$ be a HmbCA. We say that $\mathsf A$ is an {\em enriched hyper mbC algebra (EHmbCA)} if it satisfies the following additional axioms, for all $x,y,z\in A$:
    \begin{description}
        \item[E0 -] $x\in\neg\neg x$ and $x\in{\circ}{\circ} x$.
        
        \item[E1 -] $\neg x$ and $\circ x$ are stable.
        
        \item[E2 -] The following relation $\sim$ is a transitive relation (which implies, considering E0, that it is an equivalence relation):
        $$x\sim y\mbox{ iff there exists }z\mbox{ such that }x,y\in\neg z\mbox{ or }x,y\in\circ z.$$
        Note that, by E1, $x\sim y$ implies $x\equiv y$. 

        \item[E3 -] If $x\curlyvee y\equiv\top$ and $x\curlywedge y\curlywedge z\equiv\bot$ then there exists $w$ such that $x\sim w$, $y\sim\neg w$ and $z\sim\circ w$.\footnote{Remember that, for $X,Y\subseteq\mathsf A$, $X\sim Y$ denotes that $x\sim y$ for all $x\in X$ and all $y\in Y$.}
        \item [E4 -] If $x\sim y$, $\neg x\sim\neg y$ and $\circ x\sim\circ y$ then $x=y$.
    \end{description}
    For $x\in A$ we denote $[x]:=\{y \ : \ x\sim y\}$ and
    $$\mathsf U^{\mathsf{mbC}}(\mathsf A):=A/{\sim}=\{[x] \ : \ x\in A\}.$$
    The category \textbf{EHmbCA} is the one where the objects are EHmbCAs and the morphisms are the \textbf{HmbCA}-morphisms $f:\mathsf A_1\rightarrow \mathsf A_2$.
\end{definition}

It is worth noting that enriched hyper mbC algebras abstract the notion of hyper swap structures for  mbC, as the following result shows:

\begin{proposition}
    Let $\mathsf A$ be a HBA. Then the hyper swap $S^{\mathsf{mbC}}(\mathsf A)$ is an enriched hyper mbC algebra.
\end{proposition}
\begin{proof}
    Let $z\in S^{\mathsf{mbC}}(\mathsf A)$. Since $\neg z=\{w\in S^{\mathsf{mbC}}(\mathsf A) \ : \ w_1=z_2\}$ and $(z_2,z_1,z_3)\in \neg z$, we get $z=(z_1,z_2,z_3)\in \neg\neg z$. Similarly $z\in{\circ}{\circ} z$. Moreover, from
    \begin{align*}
        \neg z&=\{w\in S^{\mathsf{mbC}}(\mathsf A) \ : \ w_1=z_2\}\\
        \circ z&=\{u\in S^{\mathsf{mbC}}(\mathsf A) \ : \ u_1=z_3\}
    \end{align*}
    and $z\equiv w$ iff $z_1\equiv w_1$, we get $\neg z$ and $\circ z$ are stable, which complete the proof of E0 and E1. Now, from the very description of $\neg$ and $\circ$ we get that $z\sim w$ in $S^{\mathsf{mbC}}(\mathsf A)$ iff $z_1=w_1$, which provides E2. 

    For E3, let $x,y,z\in \mathsf S^{\mathsf{mbC}}(\mathsf A)$ with $x\curlyvee y\equiv\top$ and $x\curlywedge y\curlywedge z\equiv\bot$. This occurs iff $x_1\curlyvee y_1\equiv\top$ and $x_1\curlywedge y_1\curlywedge z_1\equiv\bot$. Picking $w=(x_1,y_1,z_1)\in \mathsf A\times \mathsf A\times \mathsf A$, we get $w\in \mathsf S^{\mathsf{mbC}}(\mathsf A)$, $x\sim w$, $y\sim\neg w$ and $z\sim\circ w$.

    Finally, to prove E4, suppose that $x\sim y$, $\neg x\sim\neg y$ and $\circ x\sim\circ y$. Since $z\sim w$ in $S^{\mathsf{mbC}}(\mathsf A)$ iff $z_1=w_1$, $x\sim y$ implies that $x_1=y_1$. Moreover $\neg x\sim\neg y$ implies that $x_2=y_2$ and $\circ x\sim\circ y$ implies that $x_3=y_3$, providing $x=y$.
\end{proof}

\begin{proposition} \label{functor U-mbC}
    Let $\mathsf A$ be an EHmbCA. Then $\langle\mathsf U^{\mathsf{mbC}}(\mathsf A),\preceq_{\mathsf U^{\mathsf{mbC}}(\mathsf A)}\rangle$ is a HBA (which will be also denoted by $\mathsf U^{\mathsf{mbC}}(\mathsf A)$) with the relation $\preceq_{\mathsf U^{\mathsf{mbC}}(\mathsf A)}$ defined by $[x]\preceq_{\mathsf U^{\mathsf{mbC}}(\mathsf A)}[y]$ iff $x\preceq y$. Moreover, the assignment $\mathsf A\mapsto \mathsf U^{\mathsf{mbC}}(\mathsf A)$ provides a functor $\mathsf U^{\mathsf{mbC}}:\textbf{EHmbCA}\rightarrow \textbf{HBA}$.
\end{proposition}
\begin{proof}
    Note that, since $\neg z$ and $\circ z$ are stable for all $z\in A$, $x\sim x'$ and $y\sim y'$ implies that $x\equiv x'$ and $y\equiv y'$. This implies that $x\preceq y$ iff $x'\preceq y'$, and $\preceq_{\mathsf U^{\mathsf{mbC}}(\mathsf A)}$ is well-defined. To obtain the HBA structure for $\mathsf U^{\mathsf{mbC}}(\mathsf A)$, just observe that
    \begin{align*}
        [x]\curlywedge_{\mathsf (\mathsf A)}[y]&:=\{[z] \ : \ z\in x\curlywedge y\}\\
        [x]\curlyvee_{\mathsf U(\mathsf A)}[y]&:=\{[z] \ : z\in x\curlyvee y\}
    \end{align*}
    are well-defined non-empty sets, which constitute the basis for a HBA structure over  $\langle\mathsf U^{\mathsf{mbC}}(\mathsf A),\preceq_{\mathsf U^{\mathsf{mbC}}(\mathsf A)}\rangle$. To finish the proof and to obtain a functor, let $f:\mathsf A_1\rightarrow \mathsf A_2$ be an \textbf{EHmbCA}-morphism and define $\mathsf U^{\mathsf{mbC}}(f):\mathsf U^{\mathsf{mbC}}(\mathsf A_1)\rightarrow \mathsf U^{\mathsf{mbC}}(\mathsf A_2)$ given by the rule $\mathsf U^{\mathsf{mbC}}(f)([x])=[f(x)]$. If $x\sim x'$ then $x,x'\in\neg w$ for some $w\in A_1$ (or $x,x'\in\circ z$ for some $z\in A_1$), which implies that $f(x),f(x')\in\neg f(w)$ (or $f(x),f(x')\in\circ f(z)$) and so $f(x)\sim f(x')$. Hence, $[f(x)]=[f(x')]$ and $\mathsf U^{\mathsf{mbC}}(f)$ is well-defined.
    
    After that, the fact that $\mathsf U^{\mathsf{mbC}}(1_{\mathsf A_1})=1_{\mathsf U^{\mathsf{mbC}}(\mathsf A_1)}$ and $\mathsf U^{\mathsf{mbC}}(f\circ g)=\mathsf U^{\mathsf{mbC}}(f)\circ \mathsf U^{\mathsf{mbC}}(g)$ is a straightforward calculation involving the above definitions.
\end{proof}

Observe that, since the hyper swap structure $S^{\mathsf{mbC}}(\mathsf A)$ over a HBA $\mathsf A$ is a EHmbCA, the Kalman functor $\mathsf S^{\mathsf{mbC}}:\textbf{HBA}\rightarrow \textbf{HmbCA}$ can be seen as a functor $\mathsf S^{\mathsf{mbC}}:\textbf{HBA}\rightarrow \textbf{EHmbCA}$. 

\begin{theorem}\label{equivalence-01-mbC}
    For all $\mathsf A\in \textbf{HBA}$ there is an isomorphism $\Phi_{\mathsf A}:\mathsf A\rightarrow \mathsf U^{\mathsf{mbC}}(\mathsf S^{\mathsf{mbC}}(\mathsf A))$. Moreover, this provides a natural isomorphism of functors $\Phi:1_{\textbf{HBA}}\Rightarrow \mathsf U^{\mathsf{mbC}}\circ \mathsf S^{\mathsf{mbC}}$ given by $\mathsf A\mapsto\Phi_{\mathsf A}$.
\end{theorem}
\begin{proof}
    The proof follows the same logical structure as Theorem 5.11 in \cite{con:gol:rob:2026}, but adapted to the context of HBAs.
    
    For all $x\in A$ there exist $y,z\in A$ such that $x\curlyvee y\equiv\top$ and $x\curlywedge y\curlywedge z\equiv\bot$ (for example, take $y\in\neg x$ and any $z\in \bot$). Define $\Phi_{\mathsf A}(x):=[(x,y,z)]$ with $y,z\in A$ such that $x\curlyvee y\equiv\top$ and $x\curlywedge y\curlywedge z\equiv\bot$. The function  $\Phi_{\mathsf A}$ is well-defined: indeed, let $y,y',z,z'\in A$ such that $x\curlyvee y\equiv\top$, $x\curlywedge y\curlywedge z\equiv\bot$, $x\curlyvee y'\equiv\top$ and $x\curlywedge y'\curlywedge z'\equiv\bot$. Then, $(x,y,z),(x,y',z')$ are elements of $S^{\mathsf{mbC}}(\mathsf A)$ such that $(x,y,z)\sim(x,y',z')$ and so $[(x,y,z)]=[(x,y',z')]$. This proves that  $\Phi_{\mathsf A}$ is well-defined. 
    
    Now, let $x,x'\in A$ and $d\in x\curlywedge x'$. Also let $d',d''\in A$ such that $d\curlyvee d'\equiv\top$ and $d\curlywedge d'\curlywedge d''\equiv\bot$. Observe that for all $y,y',z,z'\in A$ with $x\curlyvee y\equiv\top$, $x\curlywedge y\curlywedge z\equiv\bot$, $x'\curlyvee y'\equiv\top$ and $x'\curlywedge y'\curlywedge z'\equiv\bot$, we have $(d,d',d'')\in(x,y,z)\curlywedge(x',y',z')$ in $\mathsf S^{\mathsf{mbC}}(\mathsf A)$, which implies that $[(d,d',d'')]\in[(x,y,z)]\curlywedge[(x',y',z')]$ (in $\mathsf U^{\mathsf{mbC}}(\mathsf S^{\mathsf{mbC}}(\mathsf A))$). Then, we have that $d\in x\curlywedge x'$ implies that $\Phi_A(d)\in \Phi_A(x)\curlywedge \Phi_A(x')$. Similarly, we prove that $d\in x\curlyvee x'$ implies that $\Phi_A(d)\in \Phi_A(x)\curlyvee \Phi_A(x')$,  proving that $\Phi_{\mathsf A}$ is a \textbf{HBA}-morphism. Note that it is a surjective morphism: if $[(x,y,z)]\in \mathsf U^{\mathsf{mbC}}(\mathsf S^{\mathsf{mbC}}(\mathsf A))$, then $[(x,y,z)]=\Phi_{\mathsf A}(x)$.

    Finally, suppose that $\Phi_{\mathsf A}(x)=\Phi_{\mathsf a}(x')$. This means that, for some $y,y',z,z'\in A$ with $x\curlyvee y\equiv\top$, $x\curlywedge y\curlywedge z\equiv\bot$, $x'\curlyvee y'\equiv\top$ and $x'\curlywedge y'\curlywedge z'\equiv\bot$, we have $[(x,y,z)]=[(x',y',z')]$. Then, $(x,y,z)\sim(x',y',z')$, implying that $(x,y,z),(x',y',z')\in\neg w$ (or $(x,y,z),(x',y',z')\in\circ w$) for some $w\in \mathsf S^{\mathsf{mbC}}(\mathsf A)$. In particular $x=x'$, showing that $\Phi_{\mathsf A}$ is injective.

    Therefore $\Phi_{\mathsf L}:\mathsf A\rightarrow \mathsf U^{\mathsf{mbC}}(S^{\mathsf{mbC}}(\mathsf A))$ is an isomorphism which is natural in the sense that for a \textbf{HBA}-morphism $f:\mathsf A_1\rightarrow \mathsf A_2$, the following diagram commutes: 
    $$\xymatrix@!=4.5pc@R=2.5cm @C=3cm{
    \mathsf A_1\ar[r]^{\Phi_{\mathsf A_1}}\ar[d]_{f} & \mathsf U^{\mathsf{mbC}}(\mathsf S^{\mathsf{mbC}}(\mathsf A_1))\ar[d]^{\mathsf U^{\mathsf{mbC}}(\mathsf S^{\mathsf{mbC}}(f))} \\
    \mathsf A_2\ar[r]_{\Phi_{\mathsf A_2}} & \mathsf U^{\mathsf{mbC}}(\mathsf S^{\mathsf{mbC}}(\mathsf A_2))
    }$$
    In fact, for all $x\in A_1$ and all $y,z\in A_1$ with $x\curlyvee y\equiv\top$ and $x\curlywedge y\curlywedge z\equiv\bot$ we have\\
    
    $\begin{array}{ll}
         \mathsf U^{\mathsf{mbC}}(\mathsf S^{\mathsf{mbC}}(f))(\Phi_{\mathsf A_1}(x))&=\mathsf U^{\mathsf{mbC}}(\mathsf S^{\mathsf{mbC}}(f))([(x,y,z)])  \\
         &=[(f(x),f(y),f(z))] =\Phi_{\mathsf A_2}(f(x)).
    \end{array}$\\
  
\noindent    Then $\Phi:\mathsf A\rightarrow\Phi_{\mathsf A}$ is a natural isomorphism that witnesses the isomorphism of functors $$\Phi:1_{\textbf{HBA}}\cong \mathsf U^{\mathsf{mbC}}\circ \mathsf S^{\mathsf{mbC}}.$$
\end{proof}

\begin{theorem}\label{equivalence-02-mbC}
    For all $\mathsf A\in \textbf{EHmbCA}$ there is an isomorphism $\Psi_{\mathsf A}:\mathsf A\rightarrow \mathsf S^{\mathsf{mbC}}(\mathsf U^{\mathsf{mbC}}(\mathsf A))$. Moreover, this provides a natural isomorphism of functors $\Psi:1_{\textbf{EHmbCA}}\Rightarrow \mathsf S^{\mathsf{mbC}}\circ \mathsf U^{\mathsf{mbC}}$ given by $\mathsf A\mapsto\Psi_{\mathsf A}$. 
\end{theorem}
\begin{proof}
    Given $x \in A$ , if $y,y'\in\neg x$ and $z,z' \in \circ x$ then, by definition, $y \sim y'$ and $z \sim z'$, therefore $[y]=[y']$ and $[z]=[z']$. Moreover, $x \curlyvee y \equiv \top$ and $x\curlywedge y\curlywedge z\equiv\bot$, which imply $[x] \curlyvee_{\mathsf U^{\mathsf{mbC}}(\mathsf A)} [y] \equiv \top_{\mathsf U^{\mathsf{mbC}}(\mathsf A)}$ and $[x]\curlywedge_{\mathsf U^{\mathsf{mbC}}(\mathsf A)} [y]\curlywedge_{\mathsf U^{\mathsf{mbC}}(\mathsf A)} [z]\equiv\bot_{\mathsf U^{\mathsf{mbC}}(\mathsf A)}$, showing that $([x],[y],[z]) \in S^{\mathsf{mbC}}(\mathsf U^{\mathsf{mbC}}(\mathsf A))$, the domain of $\mathsf S^{\mathsf{mbC}}(\mathsf U^{\mathsf{mbC}}(\mathsf A))$. With these considerations, we have a well-defined function $\Psi_{\mathsf A}:A\rightarrow \mathsf S^{\mathsf{mbC}}(\mathsf U^{\mathsf{mbC}}(\mathsf A))$ given, for $x\in A$, by the rule $\Psi_{\mathsf A}(x):=([x],[y],[z])$, where $y\in\neg x$ and $z\in\circ x$ are arbitrary.
    
    To prove that $\Psi_{\mathsf A}$ is an \textbf{EHmbCA}-morphism, let $x,x'\in A$ and $d\in x\curlywedge x'$. Also let $d',d''\in A$ such that $d\curlyvee d'\equiv\top$ and $d\curlywedge d'\curlywedge d''\equiv\bot$. Observe that for all $y,y',z,z'\in A$ with $x\curlyvee y\equiv\top$, $x\curlywedge y\curlywedge z\equiv\bot$, $x'\curlyvee y'\equiv\top$ and $x'\curlywedge y'\curlywedge z'\equiv\bot$, we have $([d],[d'],[d''])\in([x],[y],[z])\curlywedge([x'],[y'],[z'])$ in $\mathsf S^{\mathsf{mbC}}(\mathsf U^{\mathsf{mbC}}(\mathsf A))$, which means that $\Psi_{\mathsf A}(d)\in \Psi_{\mathsf A}(x)\curlywedge \Psi_{\mathsf A}(x')$. Similarly we prove that $d\in x\curlyvee x'$ implies that $\Psi_{\mathsf A}(d)\in \Psi_{\mathsf A}(x)\curlyvee \Psi_{\mathsf A}(x')$, proving that $\Psi_{\mathsf A}$ is a \textbf{HBA}-morphism. The next step is to prove that $\Psi_{\mathsf A}(\neg x)\subseteq\neg\Psi_{\mathsf A}(x)$ and $\Psi_{\mathsf A}(\circ x)\subseteq\circ\Psi_{\mathsf A}(x)$. Observe first that  

    $\begin{array}{lll}
         \neg\Psi_{\mathsf A}(x)&=&\neg ([x],[y],[z])\\[1mm]
         &=& \big\{([y],[w_2],[w_3]\in S^{\mathsf{mbC}}(\mathsf U^{\mathsf{mbC}}(\mathsf A)) \ : \ [w_2],[w_3]\in\mathsf U^{\mathsf{mbC}}(\mathsf A)\big\}\\[1mm]
         &=& \big\{([y],[w_2],[w_3]\in S^{\mathsf{mbC}}(\mathsf U^{\mathsf{mbC}}(\mathsf A)) \ :\\[1mm]
         && \hspace{3cm} y, w_2,w_3\in\mathsf A\mbox{ with }y\in\neg x\big\}.    
    \end{array}$
    
    Now, let $z\in\neg x$. Then by definition, $\Psi_A(z)=([z],[z'],[z''])$ with $z'\in\neg z$ and $z''\in\circ z$. Then\\
    
    $\begin{array}{lll}
        \Psi_A(z)&=&([z],[z'],[z''])\in\big\{([y],[w_2],[w_3]\in S^{\mathsf{mbC}}(\mathsf U^{\mathsf{mbC}}(\mathsf A)) \ :\\[1mm]
         && \hspace{4cm} y, w_2,w_3\in\mathsf A\mbox{ with }y\in\neg x\big\}\\[1mm]
         &=& \neg\Psi_{\mathsf A}(x),
    \end{array}$\\
    
   \noindent proving that $\Psi_A(\neg x)\subseteq\neg\Psi_A(x)$. Similarly we get $\Psi_A(\circ x)\subseteq\circ\Psi_A(x)$ proving that $\Psi_{\mathsf A}:\mathsf A\rightarrow \mathsf S^{\mathsf{mbC}}(\mathsf U^{\mathsf{mbC}}(\mathsf A))$ is an \textbf{EHmbCA}-morphism.

    If $\Psi_{\mathsf A}(x)=\Psi_{\mathsf A}(y)$ then $([x],[x'],[x''])=([y],[y'],[y''])$ for all $x'\in\neg x$, $x''\in\circ x$, and all $y'\in\neg y$, $y''\in\circ y$, which means $x\sim y$, $\neg x\sim \neg y$ and $\circ x\sim \circ y$. By Axiom E4 we have $x=y$ and so $\Psi_{\mathsf A}$ is injective. Also, if $([x],[y],[z])\in S^{\mathsf{mbC}}(\mathsf U^{\mathsf{mbC}}(\mathsf A))$ then $[x]\curlyvee_{\mathsf U(\mathsf A)} [y] \equiv \top_{\mathsf U(\mathsf A)}$ and $[x]\curlywedge_{\mathsf U(\mathsf A)} [y]\curlywedge_{\mathsf U(\mathsf A)}[z] \equiv \bot_{\mathsf U(\mathsf A)}$, which implies that $x\curlyvee y\equiv\top$ and $x\curlywedge y\curlywedge z\equiv\bot$. By Axiom E3 there exist $w\in A$ such that $x\sim w$, $y\sim \neg w$ and $z\sim\circ w$. Therefore, given $w'\in\neg w$ and $w''\in\circ w$ we get $\Psi_{\mathsf A}(w)=([w],[w'],[w''])=([x],[y],[z])$ and so $\Psi_{\mathsf A}$ is surjective.
    
    Therefore $\Psi_{\mathsf A}:\mathsf A\rightarrow \mathsf S^{\mathsf{mbC}}(\mathsf U^{\mathsf{mbC}}(\mathsf A))$ is an isomorphism which is natural in the sense that for an \textbf{EHmbCA}-morphism $g:\mathsf A_1\rightarrow \mathsf A_2$, the following diagram commutes: 
    $$\xymatrix@!=4.5pc@R=2.5cm @C=3cm{
    \mathsf A_1\ar[r]^{\Psi_{\mathsf A_1}}\ar[d]_{g} & \mathsf S^{\mathsf{mbC}}(\mathsf U^{\mathsf{mbC}}(\mathsf A_1))\ar[d]^{\mathsf S^{\mathsf{mbC}}(\mathsf U^{\mathsf{mbC}}(g))} \\
    \mathsf A_2\ar[r]_{\Psi_{\mathsf A_2}} & \mathsf S^{\mathsf{mbC}}(\mathsf U^{\mathsf{mbC}}(\mathsf A_2))
    }$$
    In fact, for all $x\in A_1$, $y\in\neg x$ and $z\in\circ x$ we have
    \begin{align*}
        \mathsf S^{\mathsf{mbC}}(\mathsf U^{\mathsf{mbC}}(g))(\Psi_{\mathsf A_1}(x))&=\mathsf S^{\mathsf{mbC}}(\mathsf U^{\mathsf{mbC}}(g))([x],[y],[z])\\
        &=(\mathsf U^{\mathsf{mbC}}(g)([x]),\mathsf U^{\mathsf{mbC}}(g)([y]),\mathsf U^{\mathsf{mbC}}(g)([z]))\\
        &=([g(x)],[g(y)],[g(z)])=\Psi_{A_2}(g(x)).
    \end{align*}
    Then $\Psi:\mathsf A\rightarrow\Psi_{\mathsf A}$ is a natural isomorphism that witnessess the isomorphism of functors $$\Psi:1_{\textbf{EHmbCA}}\cong \mathsf S^{\mathsf{mbC}}\circ \mathsf U^{\mathsf{mbC}}.$$
\end{proof}

Combining Theorems~\ref{equivalence-01-mbC} and~\ref{equivalence-02-mbC} we arrive at the main result in this section:

\begin{theorem}\label{equivalence-03-mbC}
    The functors $\mathsf S^{\mathsf{mbC}}:\textbf{HBA}\rightarrow \textbf{EHmbCA}$ and $\mathsf U^{\mathsf{mbC}}:\textbf{EHmbCA}\rightarrow \textbf{HBA}$ establish an equivalence of categories.
\end{theorem}

The latter result shows that any  enriched hyper mbC algebra has a representation as a swap structure over a hyper Boolean algebra.

\section{Extensions of these results for others LFIs}\label{sec7}

This section demonstrates the versatility of our hyperalgebraic framework by extending the results obtained for mbC to a wider class of Logics of Formal Inconsistency (LFIs), including mbCciw, mbCci, and Ci. Just as we adapted the swap structures for extensions of $C_\omega$ (such as $C_{min}$ and $C_\omega^+$) in~\cite{con:gol:rob:2026}, here we show that the category of Hyper Boolean Algebras can be naturally restricted to subcategories corresponding to these axiomatic extensions. By imposing additional algebraic conditions that mirror the logical axioms---such as those relating the consistency operator $\circ$ to non-contradiction---we derive specific classes of hyperalgebras and their corresponding hyper swap structures. This establishes a comprehensive algebraic semantics for the hierarchy of LFIs described in~\cite{carnielli2016paraconsistent}, confirming that the equivalence of categories is preserved across these systems.

\subsection{ Axiomatic Extensions for LFIs}

Before defining the algebraic counterparts, we briefly recall the Hilbert calculi for the relevant extensions of mbC. Following the taxonomy presented in~\cite[Chapter 3]{carnielli2016paraconsistent}, we focus on systems that strengthen the consistency operator $\circ$ or the paraconsistent negation $\neg$, specifically mbCciw, mbCci, Ci, Cie, and Cia.

\begin{definition}
    The logic mbCciw is obtained from mbC by adding the axiom schema
    \begin{align*}
        \tag{ciw}\circ\alpha\vee(\alpha\wedge\neg\alpha)
    \end{align*}
\end{definition}

\begin{definition}
    The logic mbCci is obtained from mbC by adding the axiom schema
    \begin{align*}
        \tag{ci}\neg{\circ}\alpha\rightarrow(\alpha\wedge\neg\alpha)
    \end{align*}
\end{definition}

\begin{definition}
    The logic Ci is obtained from mbCci by adding the axiom schema
    \begin{align*}
        \tag{cf}\neg\neg\alpha\rightarrow\alpha
    \end{align*}
\end{definition}

\begin{definition}
    The logic Cie is obtained from Ci by adding the axiom schema
    \begin{align*}
        \tag{Cie}\alpha\rightarrow\neg\neg\alpha
    \end{align*}
\end{definition}

\begin{definition}
    The logic Cia is obtained from Ci by adding the axiom schemas, for $\#\in\{\wedge,\vee,\rightarrow\}$:
    \begin{align*}
        \tag{Cia$_\#$}(\circ\alpha\wedge\circ\beta)\rightarrow\circ(\alpha\#\beta)
    \end{align*}
\end{definition}

\subsection{The Hyper Algebras and Hyper Swaps}

In this subsection, we introduce the classes of hyperalgebras corresponding to the logics defined above. By imposing additional algebraic restrictions on the \emph{Hyper mbC Algebras} (HmbCAs) defined in Section~\ref{sec5}, we obtain specific structures such as \emph{Hyper mbCciw Algebras} and \emph{Hyper Ci Algebras}. Mirroring the approach taken for $C_{min}$ in~\cite{con:gol:rob:2026}, each logical axiom is translated into a corresponding condition on the hyperoperations, ensuring that the resulting Hyper Swap Structures $S^{\mathcal{L}}(B)$ naturally inhabit these subcategories.

\begin{definition}\label{hyper-mbCciw-algebra}
    A {\em hyperalgebra for mbCciw} (or {\em hyper  mbCciw algebra}, or simply a HmbCciwA) is a HmbCA such that for every $x\in A$:

\begin{description}
    \item[(ciw)] $(x\curlywedge y)\curlyvee z\equiv\top$ for all $y \in \neg x$ and all $z\in\circ x$.
\end{description}
\noindent It is immediate to see that condition (ciw) can be written in a concise way as follows:

\begin{description}
    \item[(ciw')] $(x\curlywedge \neg x)\curlyvee\circ x\equiv\top$ for all $x\in A$;
\end{description}

or equivalently,

\begin{description}
    \item[(ciw'')] $\circ x\equiv -(x\curlywedge \neg x)$ for all $x\in A$.
\end{description}
\end{definition}

\begin{definition} [Hyper Swap structures for mbCciw] \label{def-hyper-swap-mbcciw} \ \\
Let $\mathsf B=\langle B, \curlywedge,\curlyvee,-, \top,\bot\rangle$ be a HBA. Let 
$$S^{\mathsf{mbCciw}}_{\mathsf B}=\{z \in B \times B\times B \ : \ z_1 \curlyvee z_2\equiv\top\mbox{ and } z_3\equiv -(z_1\curlywedge\neg z_2)\}.$$
The {\em hyper swap structure} for $mbCciw$ over $\mathsf B$ is the hyperalgebra $\mathsf S^{\mathsf{mbCciw}}(\mathsf B)=\langle S^{\mathsf{mbCciw}}_{\mathsf B}, \curlywedge,\curlyvee,\multimap,\neg,\circ\rangle$ over the signature $\Sigma$ such that the hyperoperators are defined in Definition \ref{def-hyper-swap-mbc}.\footnote{By simplicity, the hyperoperators and the induced preordering in $\mathsf S^{\mathsf{mbCciw}}(\mathsf B)$ will be denoted by using the same symbols than in $\mathsf B$. The context will avoid any confusion. Also, we will do this simplification in all the hyper swap for the others LFIs.} We may denote $S^{\mathsf{mbCciw}}(B)$ simply by  $S^{\mathsf{mbCciw}}_{\mathsf B}$.
\end{definition}

For the convenience of the reader, we recall here the more succint description of hyperoperations in  $S^{\mathsf{mbCciw}}_{\mathsf B}$ (which is the same as in $S^{\mathsf{mbC}}_{\mathsf B}$):
\begin{align*}
    z\curlywedge w&=(z_1\curlywedge w_1, \, \_,\, \_)\\
    z\curlyvee w&=(z_1\curlyvee w_1, \, \_,\, \_)\\
    z\multimap w&=(z_1\multimap w_1, \, \_,\, \_)\\
    \neg z&=(z_2, \, \_,\, \_)\\
    \circ z&=(z_3, \, \_,\, \_)
\end{align*}

\begin{definition}\label{hyper-mbcci-algebra}
    A {\em hyperalgebra for mbCci} (or {\em hyper  mbCci algebra}, or simply a HmbCciA) is a HmbCciwA such that for every $x\in A$:

\begin{description}
    \item[(ci)] $\neg{\circ} x\equiv x\curlywedge\neg x$ for all $x\in A$.
\end{description}
\end{definition}

\begin{definition} [Hyper Swap structures for mbCci] \label{def-hyper-swap-mbcci}
Let $\mathsf B=\langle B, \curlywedge,\curlyvee,-, \top,\bot\rangle$ be a HBA. Let 
$$S^{\mathsf{mbCci}}_{\mathsf B}=\{z \in B \times B\times B \ : \ z_1 \curlyvee z_2\equiv\top\mbox{ and } z_3\equiv -(z_1\curlywedge\neg z_2)\}.$$
The {\em hyper swap structure} for $mbCci$ over $\mathsf B$ is the hyperalgebra $\mathsf S^{\mathsf{mbCci}}(\mathsf B)=\langle S^{\mathsf{mbCci}}_{\mathsf B}, \curlywedge,\curlyvee,\multimap,\neg,\circ\rangle$ over the signature $\Sigma$ such that the hyperoperators $\curlywedge,\curlyvee,\multimap$ and $\neg$ are defined in the same way as in Definition \ref{def-hyper-swap-mbc}. The $\circ$ in $\mathsf S^{\mathsf{mbCci}}(\mathsf B)$ is defined by:
\begin{align*}
    \circ z&:=\{u\in S^{\mathsf{mbCci}}_{\mathsf B} \ : \ u_1=z_3\mbox{ and }u_2\equiv z_1\curlywedge z_2\}.
\end{align*}
 We may denote $S^{\mathsf{mbCci}}(B)$ simply by  $S^{\mathsf{mbCci}}_{\mathsf B}$.
\end{definition}

\begin{definition}\label{hyper-Ci-algebra}
    A {\em hyperalgebra for Ci} (or {\em hyper Ci algebra}, or simply a HCiA) is a HmbCciA such that for every $x\in A$:

\begin{description}
    \item[(cf)] $\neg\neg x\preceq x$.
\end{description} 
\end{definition}

\begin{definition} [Hyper Swap structures for Ci] \label{def-hyper-swap-ci}
Let $\mathsf B=\langle B, \curlywedge,\curlyvee,-, \top,\bot\rangle$ be a HBA. Let 
$$S^{\mathsf{Ci}}_{\mathsf B}=\{z \in B \times B\times B \ : \ z_1 \curlyvee z_2\equiv\top\mbox{ and } z_3\equiv -(z_1\curlywedge\neg z_2)\}.$$
The {\em hyper swap structure} for $Ci$ over $\mathsf B$ is the hyperalgebra $\mathsf S^{\mathsf{Ci}}(\mathsf B)=\langle S^{\mathsf{Ci}}_{\mathsf B}, \curlywedge,\curlyvee,\multimap,\neg,\circ\rangle$ over the signature $\Sigma$ such that the hyperoperators $\curlywedge,\curlyvee$ and $\multimap$ are defined in the same way as in Definition \ref{def-hyper-swap-mbc}. The hyperoperators $\neg$ and $\circ$ are defined by:
\begin{align*}
    \neg z&:=\{u\in S^{\mathsf{Ci}}_{\mathsf B} \ : \ u_1=z_2\mbox{ and }u_2\preceq z_1\}\\
    \circ z&:=\{u\in S^{\mathsf{Ci}}_{\mathsf B} \ : \ u_1=z_3\mbox{ and }u_2\equiv z_1\curlywedge z_2\}.
\end{align*}
 We may denote $S^{\mathsf{Ci}}(B)$ simply by  $S^{\mathsf{Ci}}_{\mathsf B}$.
\end{definition}

\begin{definition}\label{hyper-Cie-algebra}
    A {\em hyperalgebra for Cie} (or {\em hyper Cie algebra}, or simply a HCieA) is a HCiA such that for every $x\in A$:

\begin{description}
    \item[(Cie)] $\neg\neg x\equiv x$.
\end{description} 
\end{definition}

\begin{definition} [Hyper Swap structures for Cie] \label{def-hyper-swap-cie}
Let $\mathsf B=\langle B, \curlywedge,\curlyvee,-, \top,\bot\rangle$ be a HBA. Let 
$$S^{\mathsf{Cie}}_{\mathsf B}=\{z \in B \times B\times B \ : \ z_1 \curlyvee z_2\equiv\top\mbox{ and } z_3\equiv -(z_1\curlywedge\neg z_2)\}.$$
The {\em hyper swap structure} for $Cie$ over $\mathsf B$ is the hyperalgebra $\mathsf S^{\mathsf{Cie}}(\mathsf B)=\langle S^{\mathsf{Cie}}_{\mathsf B}, \curlywedge,\curlyvee,\multimap,\neg,\circ\rangle$ over the signature $\Sigma$ such that the hyperoperators $\curlywedge,\curlyvee$ and $\multimap$ are defined in the same way as in Definition \ref{def-hyper-swap-mbc}. The hyperoperators $\neg$ and $\circ$ are defined by:
\begin{align*}
    \neg z&:=\{u\in S^{\mathsf{Cie}}_{\mathsf B} \ : \ u_1=z_2\mbox{ and }u_2\equiv z_1\}\\
    \circ z&:=\{u\in S^{\mathsf{Cie}}_{\mathsf B} \ : \ u_1=z_3\mbox{ and }u_2\equiv z_1\curlywedge z_2\}.
\end{align*}
 We may denote $S^{\mathsf{Cie}}(B)$ simply by  $S^{\mathsf{Cie}}_{\mathsf B}$.
\end{definition}

\begin{definition}\label{hyper-Cia-algebra}
    A {\em hyperalgebra for Cia} (or {\em hyper Cia algebra}, or simply a HCiaA) is a HCiA such that for every $x,y\in A$:

\begin{description}
    \item[(Cia$_\wedge$)] $\circ x\curlywedge\circ y\preceq\circ(x\curlywedge y)$.
    \item[(Cia$_\vee$)] $\circ x\curlywedge\circ y\preceq\circ(x\curlyvee y)$.
    \item[(Cia$_\rightarrow$)] $\circ x\curlywedge\circ y\preceq\circ(x\multimap y)$.
\end{description}  
\end{definition}

\begin{definition} [Hyper Swap structures for Cia] \label{def-hyper-swap-cia}
Let $\mathsf B=\langle B, \curlywedge,\curlyvee,-, \top,\bot\rangle$ be a HBA. Let 
$$S^{\mathsf{Cia}}_{\mathsf B}=\{z \in B \times B\times B \ : \ z_1 \curlyvee z_2\equiv\top\mbox{ and } z_3\equiv -(z_1\curlywedge\neg z_2)\}.$$
The {\em hyper swap structure} for $Cia$ over $\mathsf B$ is the hyperalgebra $\mathsf S^{\mathsf{Cia}}(\mathsf B)=\langle S^{\mathsf{Cia}}_{\mathsf B}, \curlywedge,\curlyvee,\multimap,\neg,\circ\rangle$ over the signature $\Sigma$ such that the hyperoperators are defined as follows:
\begin{align*}
    z\curlywedge w&:=\{u\in S^{\mathsf{Cia}}_{\mathsf B} \ : \ u_1\in z_1\curlywedge w_1\mbox{ and } z_3\curlywedge w_3\preceq u_3\}\\
    z\curlyvee w&:=\{u\in S^{\mathsf{Cia}}_{\mathsf B} \ : \ u_1\in z_1\curlyvee w_1\mbox{ and } z_3\curlywedge w_3\preceq u_3 \}\\
    z\multimap w&:=\{u\in S^{\mathsf{Cia}}_{\mathsf B} \ : \ u_1\in z_1\multimap w_1\mbox{ and } z_3\curlywedge w_3\preceq u_3 \}\\
    \neg z&:=\{u\in S^{\mathsf{Cia}}_{\mathsf B} \ : \ u_1=z_2\mbox{ and }u_2\equiv z_1\}\\
    \circ z&:=\{u\in S^{\mathsf{Cia}}_{\mathsf B} \ : \ u_1=z_3\mbox{ and }u_2\equiv z_1\curlywedge z_2\}.
\end{align*}
 We may denote $S^{\mathsf{Cia}}(B)$ simply by  $S^{\mathsf{Cia}}_{\mathsf B}$.
\end{definition}

\subsection{Soundness and Completeness}

Here we demonstrate that each logic $\mathcal{L}$ in the hierarchy is sound and complete with respect to the class of its corresponding hyperalgebras. Furthermore, we show that the Nmatrices induced by the hyper swap structures $S^{\mathcal{L}}(B)$ characterize these logics, generalizing the completeness results obtained for mbC in Theorem~\ref{Sound-comple-mbC}.

\textbf{From now on, $\mathcal L$ will denote one of the logics mbCciw, mbCci, Ci, Cie or Cia.}

\begin{definition} [$\mathcal L$ semantics] \label{def-sem-L}
Let $\mathcal L\in\{\mathsf{mbCciw}, \mathsf{mbCci},\mathsf{Ci}, \mathsf{Cie},\mathsf{Cia}\}$ and $\mathsf A$ be a H$\mathcal L$A. Let $\Gamma \cup \{\varphi\}$ be a set of formulas over $\Sigma$.
\begin{enumerate}
    \item The Nmatrix associated with $\mathsf A$ is $\mathcal M^{\mathsf{H}\mathcal L\mathsf{A}}_{\mathsf A}=\langle \mathsf A, \top\rangle$.

    \item We say that $\varphi$ is a semantical consequence of $\Gamma$ w.r.t. $\mathsf A$ if $\Gamma \models_{\mathcal M^{\mathsf{H}\mathcal L\mathsf{A}}_{\mathsf A}} \varphi$.

    \item Let $\mathsf{H}\mathcal L\mathsf{A}$ be the class of H$\mathcal L$As. Then, $\varphi$ is a semantical consequence of $\Gamma$ w.r.t.  H$\mathcal L$As, denoted by $\Gamma \models_{\mathsf{H}\mathcal L\mathsf{A}} \varphi$, if $\Gamma \models_{\mathcal M^{\mathsf{H}\mathcal L\mathsf{A}}_{\mathsf A}} \varphi$ for every $\mathsf A \in \mathsf{H}\mathcal L\mathsf{A}$.
\end{enumerate}
\end{definition}

\begin{definition} \label{def-Nmatrix-mbCciw}
 Let $\mathcal L\in\{\mathsf{mbCciw}, \mathsf{mbCci},\mathsf{Ci}, \mathsf{Cie},\mathsf{Cia}\}$ and  $\mathsf A$ be a HBA. The Nmatrix associated with $\mathsf S^{\mathcal L}(\mathsf A)$ is $\mathcal{M}^{\mathcal L}(\mathsf A)=\langle \mathsf S^{\mathcal L}(\mathsf A),D^{\mathcal L}(\mathsf A) \rangle$ where the set of designated truth-values is $D^{\mathcal L}(\mathsf A)=\{z \in \mathsf S^{\mathcal L}(\mathsf A) \ : \ z_1\in\top\}$.
\end{definition}

\begin{proposition} \label{swap-SHmbCciw}
Let $\mathcal L\in\{\mathsf{mbCciw}, \mathsf{mbCci},\mathsf{Ci}, \mathsf{Cie},\mathsf{Cia}\}$. Let $\mathsf A$ be  a HBA, and let $\mathsf S^{\mathcal L}(\mathsf A)$ be the hyper swap structure for $\mathcal L$ over $\mathsf A$. Then:
\begin{enumerate}
    \item The relation $z \preceq w$  in  $\mathsf S^{\mathcal L}(\mathsf A)$  iff  $z_1 \preceq w_1$  in $\mathsf A$ defines a  preorder such that $\mathsf S^{\mathcal L}(\mathsf A)$ is an hyperlattice where, for every $z,w \in  S^{\mathcal L}(\mathsf A)$, $z\curlywedge w$ and $z\curlyvee w$ are the infimoid and the supremoid of $z$ and $w$, respectively.
    Moreover, $z \equiv w$ in  $\mathsf S^{\mathcal L}(\mathsf A)$ iff $z_1 \equiv w_1$ in $\mathsf A$. 
    \item $\mathsf S^{\mathcal L}(\mathsf A)$ is a H$\mathcal L$A.
    Moreover, $D^{\mathcal L}(\mathsf A)=\top$.
    \item $\mathcal{M}^{\mathcal L}(\mathsf A)= \mathcal M_{\mathsf S^{\mathcal L}(\mathsf A)}$.
\end{enumerate}
\end{proposition}

\begin{definition} [Hyper Swap structures semantics for $\mathcal L$] \label{def-sem-swap-L} \ \\
Let $\mathcal L\in\{\mathsf{mbCciw}, \mathsf{mbCci},\mathsf{Ci}, \mathsf{Cie},\mathsf{Cia}\}$. Let $\Gamma \cup \{\varphi\}$ be a set of formulas over $\mathcal L$. Then, $\varphi$ is a semantical consequence of $\Gamma$ w.r.t. hyper swap structures, denoted by $\Gamma \models_{\mathcal L}^{\mathsf{HSW}} \varphi$, whenever $\Gamma \models_{\mathcal{M}(\mathsf A)} \varphi$ for every hyper Boolean algebra $\mathsf A$.
\end{definition}

\begin{theorem} [Soundness and completeness of $\mathcal L$ w.r.t. hyperstructures semantics] \label{Sound-comple-L}\ \\
Let $\mathcal L\in\{\mathsf{mbCciw}, \mathsf{mbCci},\mathsf{Ci}, \mathsf{Cie},\mathsf{Cia}\}$. Let $\Gamma \cup \{\varphi\}$ be a set of formulas over $\Sigma$. The following assertions are equivalent:
\begin{enumerate}
    \item $\Gamma \vdash_{\mathcal L} \varphi$;
    \item $\Gamma \models_{\mathsf{H}\mathcal L\mathsf{A}} \varphi$;
    \item $\Gamma \models_{\mathcal L}^{\mathsf{HSW}} \varphi$.
\end{enumerate}
\end{theorem}

\subsection{Kalman Functors}

Finally, we extend the categorical analysis to these stronger systems. To obtain a full equivalence of categories, we introduce the notion of \emph{Enriched Hyper $\mathcal{L}$ Algebras} (EH$\mathcal L$As) for each logic $\mathcal{L}$. This step is crucial to circumvent the lack of adjunction discussed in the Introduction. By refining the functors $S$ and $U$ defined in Section~\ref{sec6}, we prove that the category of Hyper Boolean Algebras is equivalent to the category of EH$\mathcal L$As, showing that the representation of these logics via swap structures is categorically robust.

\begin{definition}
    Let $\mathcal L\in\{\mathsf{mbCciw}, \mathsf{mbCci},\mathsf{Ci}, \mathsf{Cie},\mathsf{Cia}\}$. The category \textbf{H$\mathcal L$A} is the one where the objects are H$\mathcal L$As and the morphisms are the H$\mathcal L$A-morphisms.
\end{definition}

\begin{theorem}
    Let $\mathcal L\in\{\mathsf{mbCciw}, \mathsf{mbCci},\mathsf{Ci}, \mathsf{Cie},\mathsf{Cia}\}$. The hyper swap structure construction provides a Kalman functor $\mathsf S^{\mathcal L}:\textbf{HBA}\rightarrow \textbf{H$\mathcal L$A}$.
\end{theorem}

\begin{definition}[Enriched Hyper mbCciw Algebras] \label{enriched-mbCciw} \ \\
    Let $\mathsf A=\langle A,\curlywedge,\curlyvee,-,\bot,\top,\neg,\circ\rangle$ be a HmbCiwA. We say that $\mathsf A$ is an {\em enriched hyper mbCciw algebra (EHmbCciwA)} if it satisfies the following additional axioms, for all $x,y,z\in A$:
    \begin{description}
        \item[E0 -] $x\in\neg\neg x$ and $x\in{\circ}{\circ} x$.
        
        \item[E1 -] $\neg x$ and $\circ x$ are stable.
        
        \item[E2 -] The following relation $\sim$ is a transitive relation (which implies, considering E0, that it is an equivalence relation):
        $$x\sim y\mbox{ iff there exists }z\mbox{ such that }x,y\in\neg z\mbox{ or }x,y\in\circ z.$$
        Note that, by E1, $x\sim y$ implies $x\equiv y$. 

        \item[E3 -] If $x\curlyvee y\equiv\top$ and $z\equiv -(x\curlywedge \neg y)$ then there exists $w$ such that $x\sim w$, $y\sim\neg w$ and $z\sim\circ w$.\footnote{Remember that, for $X,Y\subseteq\mathsf A$, $X\sim Y$ denotes that $x\sim y$ for all $x\in X$ and all $y\in Y$.}
        \item [E4 -] If $x\sim y$, $\neg x\sim\neg y$ and $\circ x\sim\circ y$ then $x=y$.
    \end{description}
    For $x\in A$ we denote $[x]:=\{y \ : \ x\sim y\}$ and
    $$\mathsf U^{\mathsf{mbCciw}}(\mathsf A):=A/{\sim}=\{[x] \ : \ x\in A\}.$$
    The category \textbf{EHmbCciwA} is the one where the objects are EHmbCciwAs and the morphisms are the \textbf{HmbCciwA}-morphisms $f:\mathsf A_1\rightarrow \mathsf A_2$.
\end{definition}

\begin{definition}[Enriched Hyper $\mathcal L$ Algebras] \label{enriched-L}
    Let $\mathcal L\in\{\mathsf{mbCci},\mathsf{Ci}, \mathsf{Cie},\mathsf{Cia}\}$. Let $\mathsf A=\langle A,\curlywedge,\curlyvee,-,\bot,\top,\neg,\circ\rangle$ be a H$\mathcal L$A. We say that $\mathsf A$ is an {\em enriched hyper $\mathcal L$ algebra (EH$\mathcal L$A)} if it is an EHmbCciwA. For $x\in A$ we denote $[x]:=\{y \ : \ x\sim y\}$ and
    $$\mathsf U^{\mathcal L}(\mathsf A):=A/{\sim}=\{[x] \ : \ x\in A\}.$$
    The category \textbf{EH$\mathcal L$A} is the one where the objects are EH$\mathcal L$As and the morphisms are the \textbf{H$\mathcal L$A}-morphisms.
\end{definition}

It is worth noting that enriched hyper $\mathcal L$ algebras abstract the notion of hyper swap structures for  $\mathcal L$, as the following result shows:

\begin{proposition}
    Let $\mathcal L\in\{\mathsf{mbCciw}, \mathsf{mbCci},\mathsf{Ci}, \mathsf{Cie},\mathsf{Cia}\}$. Let $\mathsf A$ be a HBA. Then the hyper swap $S^{\mathcal L}(B)$ is an enriched hyper $\mathcal L$ algebra.
\end{proposition}

\begin{proposition} \label{functor U-L}
    Let $\mathcal L\in\{\mathsf{mbCciw}, \mathsf{mbCci},\mathsf{Ci}, \mathsf{Cie},\mathsf{Cia}\}$ and let $\mathsf A$ be an EH$\mathcal L$A. Then $\langle\mathsf U^{\mathcal L}(\mathsf A),\preceq_{\mathsf U^{\mathcal L}(\mathsf A)}\rangle$ is a HBA (which will be also denoted by $\mathsf U^{\mathcal L}(\mathsf A)$) with the relation $\preceq_{\mathsf U^{\mathcal L}(\mathsf A)}$ defined by $[x]\preceq_{\mathsf U^{\mathcal L}(\mathsf A)}[y]$ iff $x\preceq y$. Moreover, the assignment $\mathsf A\mapsto \mathsf U^{\mathcal L}(\mathsf A)$ provides a functor $\mathsf U^{\mathcal L}:\textbf{EH$\mathcal L$A}\rightarrow \textbf{HBA}$.
\end{proposition}

Observe that, since the hyper swap structure $S^{\mathcal L}(\mathsf A)$ over a HBA $\mathsf A$ is an EH$\mathcal L$A, the Kalman functor $\mathsf S:\textbf{HBA}\rightarrow \textbf{H$\mathcal L$A}$ can be seen as a functor $\mathsf S^{\mathcal L}:\textbf{HBA}\rightarrow \textbf{EH$\mathcal L$A}$. 

\begin{theorem}\label{equivalence-01-L}
    Let $\mathcal L\in\{\mathsf{mbCciw}, \mathsf{mbCci},\mathsf{Ci}, \mathsf{Cie},\mathsf{Cia}\}$. For all $\mathsf A\in \textbf{HBA}$ there is an isomorphism $\Phi_{\mathsf A}:\mathsf A\rightarrow \mathsf U^{\mathcal L}(\mathsf S^{\mathcal L}(\mathsf A))$. Moreover, this provides a natural isomorphism of functors $\Phi:1_{\textbf{HBA}}\Rightarrow \mathsf U^{\mathcal L}\circ \mathsf S^{\mathcal L}$ given by $\mathsf A\mapsto\Phi_{\mathsf A}$.
\end{theorem}

\begin{theorem}\label{equivalence-02-L}
    Let $\mathcal L\in\{\mathsf{mbCciw}, \mathsf{mbCci},\mathsf{Ci}, \mathsf{Cie},\mathsf{Cia}\}$. For all $\mathsf A\in \textbf{EH$\mathcal L$A}$ there is an isomorphism $\Psi_{\mathsf A}:\mathsf A\rightarrow \mathsf S^{\mathcal L}(\mathsf U^{\mathcal L}(\mathsf A))$. Moreover, this provides a natural isomorphism of functors $\Psi:1_{\textbf{EH$\mathcal L$A}}\Rightarrow \mathsf S^{\mathcal L}\circ \mathsf U^{\mathcal L}$ given by $\mathsf A\mapsto\Psi_{\mathsf A}$. 
\end{theorem}

Combining Theorems~\ref{equivalence-01-L} and~\ref{equivalence-02-L} we arrive at the main result in this section:

\begin{theorem}\label{equivalence-03-L}
    Let $\mathcal L\in\{\mathsf{mbCciw}, \mathsf{mbCci},\mathsf{Ci}, \mathsf{Cie},\mathsf{Cia}\}$. The functors $\mathsf S^{\mathcal L}:\textbf{HBA}\rightarrow \textbf{EH$\mathcal L$A}$ and $\mathsf U^{\mathcal L}:\textbf{EH$\mathcal L$A}\rightarrow \textbf{HBA}$ establish an equivalence of categories.
\end{theorem}

The latter result shows that any  enriched hyper $\mathcal L$ algebra has a representation as a swap structure over a hyper Boolean algebra.

\section{Final Remarks} \label{sec8}

In this paper, we have established a hyperalgebraic counterpart to the theory of twist structures for the Logics of Formal Inconsistency (LFIs). By introducing the concept of \textbf{Hyper Boolean Algebras (HBAs)}, we provided a natural generalization of Boolean algebras that accommodates the non-deterministic nature of paraconsistent negations found in mbC and some of its extensions.

By identifying the category of \emph{Enriched Hyper mbC Algebras} (EHmbCAs) and defining the functor $S^{\mathsf{mbC}}$ over HBAs, we establish that the swap structure construction is not merely a heuristic tool for generating Nmatrices, but a functorial operation that induces a full equivalence of categories. This result parallels Cignoli's classical representation of Nelson algebras via twist structures over Heyting algebras, effectively lifting this methodology to the realm of hyperalgebras. Furthermore, the modular adaptation of this framework to stronger LFIs---such as mbCciw, mbCci, and Ci---confirms that Hyper Boolean Algebras provide a robust semantics for this hierarchy of paraconsistent logics.

Several promising directions for future research naturally arise from this framework. First, we intend to adapt this methodology to other classes of non-classical logics, such as Logics of Evidence and Truth (LETs) and Ivlev-like modal logics. A significant challenge in this direction is the absence of a natural hyperlattice structure for the hyper-swap representations of these logics---a gap that will inevitably require a generalization of the very notion of hyperlattice.

Second, we aim to investigate the connections between our hyperalgebraic approach and standard algebraization methods. Specifically, a work in progress consists of replicating the foundational ideas of Blok-Pigozzi algebraizability~\cite{blok1989algebraizable}, or even the broader notion of order-algebraizability developed by Raftery~\cite{raftery2013order}, within the context of hyper swap structures.

Third, the categories of hyperlattices themselves warrant a deeper algebraic and categorical analysis, independent of their logical applications. Finally, a challenging open problem remains: finding an intrinsic axiomatic description of the categories of enriched hyperalgebras. As noted in our analysis, the class of EH$\mathcal L$A (and even its analogue for $C_\omega$) is currently defined constructively as the image of the swap functor; a purely axiomatic characterization would complete the (hyper)algebraic picture of these structures.


\

\noindent {\bf Acknowledgments:}
Coniglio acknowledges support by an individual research grant from the National Council for Scientific and Technological Development (CNPq, Brazil), grant 309830/2023-0. All the authors were supported by the S\~ao Paulo Research Foundation (FAPESP, Brazil), thematic project {\em Rationality, logic and probability -- RatioLog}, grant  2020/16353-3. Roberto was supported by a post-doctoral grant from FAPESP, grant 2024/18577-7.

\bibliographystyle{plain}

\end{document}